\documentclass[12pt,a4paper]{amsart}

\usepackage{amssymb}
\usepackage{amscd}
\usepackage{hyperref}
\usepackage[numeric, abbrev, nobysame]{amsrefs}

\def\frak{\mathfrak}
\def\Bbb{\mathbb}
\def\Cal{\mathcal}

\let\phi\varphi

\newcommand{\x}{\times}
\renewcommand{\o}{\circ}

\newcommand{\al}{\alpha}
\newcommand{\be}{\beta}

\newcommand{\de}{\delta}

\newcommand{\ka}{\kappa}
\newcommand{\la}{\lambda}
\newcommand{\om}{\omega}
\newcommand{\ph}{\phi}
\newcommand{\ps}{\psi}
\renewcommand{\th}{\theta}
\newcommand{\si}{\sigma}

\newcommand{\Ga}{\Gamma}
\newcommand{\La}{\Lambda}
\newcommand{\Ph}{\Phi}
\newcommand{\Ps}{\Psi}
\newcommand{\Om}{\Omega}

\newcommand{\im}{\operatorname{im}}
\renewcommand{\Im}{\operatorname{im}}

\newcommand{\Coker}{\operatorname{coker}}
\newcommand{\id}{\operatorname{id}}
\newcommand{\ad}{\operatorname{ad}}
\newcommand{\Ad}{\operatorname{Ad}}
\newcommand{\Aut}{\operatorname{Aut}}

\newcommand{\hor}{\text{hor}}

\newcommand{\tcg}{{\tilde{\Cal G}}}
\newcommand{\tcw}{{\tilde{\Cal W}}}
\newcommand{\tbw}{{\tilde{\Bbb W}}}

\newcounter{theorem}
\numberwithin{theorem}{section}
\newtheorem{thm}[theorem]{Theorem}
\newtheorem*{thm*}{Theorem \thesubsection}
\newtheorem{lemma}[theorem]{Lemma}
\newtheorem{prop}[theorem]{Proposition}
\newtheorem{cor}[theorem]{Corollary}
\newtheorem*{lemma*}{Lemma \thesubsection}
\newtheorem*{prop*}{Proposition \thesubsection}
\newtheorem*{cor*}{Corollary \thesubsection}

\theoremstyle{definition}
\newtheorem{definition}[theorem]{Definition}
\newtheorem*{definition*}{Definition \thesubsection}

\newtheorem*{example*}{Example \thesubsection}
\theoremstyle{remark}
\newtheorem{remark}[theorem]{Remark}
\newtheorem*{remark*}{Remark \thesubsection}

\def\sideremark#1{\ifvmode\leavevmode\fi\vadjust{\vbox to0pt{\vss
 \hbox to 0pt{\hskip\hsize\hskip1em
 \vbox{\hsize3cm\tiny\raggedright\pretolerance10000
  \noindent #1\hfill}\hss}\vbox to8pt{\vfil}\vss}}}%
                        
\begin{document}

\title{Parabolic conformally symplectic structures III;\\
Invariant Differential operators and complexes} 
\date{December 30, 2017}
\author{Andreas \v Cap and Tom\'a\v s Sala\v c}
\thanks{Support by projects P23244--N13 (both authors) and P27072--N25 
  (first author) of the Austrian Science Fund (FWF) and by grant
  17--01171S of the Grant Agency of the Czech Republic (GA\v CR) is
  gratefully acknowledged.}

\address{A.\v C.: Faculty of Mathematics\\
University of Vienna\\
Oskar--Morgenstern--Platz 1\\
1090 Wien\\
Austria}
\address{T.S.: Mathematical Institute\\ Charles University\\ Sokolovsk\'a
  83\\Praha\\Czech Republic}
\email{Andreas.Cap@univie.ac.at}
\email{salac@karlin.mff.cuni.cz}

\begin{abstract}
  This is the last part of a series of articles on a family of
  geometric structures (PACS--structures) which all have an underlying
  almost conformally symplectic structure. While the first part of the
  series was devoted to the general study of these structures, the
  second part focused on the case that the underlying structure is
  conformally symplectic (PCS--structures). In that case, we obtained
  a close relation to parabolic contact structures via a concept of
  parabolic contactification. It was also shown that special
  symplectic connections (and thus all connections of exotic
  symplectic holonomy) arise as the canonical connection of such a
  structure.

  In this last part, we use parabolic contactifications and
  constructions related to Bernstein--Gelfand--Gelfand (BGG) sequences
  for parabolic contact structures, to construct sequences of
  differential operators naturally associated to a PCS--structure. In
  particular, this gives rise to a large family of complexes of
  differential operators associated to a special symplectic
  connection. In some cases, large families of complexes for more
  general instances of PCS--structures are obtained. 
\end{abstract}


\maketitle

\pagestyle{myheadings} \markboth{\v Cap and Sala\v c}{PCS--structures
  III}
 
\section{Introduction}\label{1}

This article is the last part in a series of three which aims at
constructing a large family of differential complexes naturally
associated to certain geometric structures. These structures are
associated to certain parabolic subalgebras in simple Lie algebras and
they come with an underlying almost conformally symplectic structure,
so we call them \textit{parabolic almost conformally symplectic
  structures} or PACS--structures for short. The precise definition of
these structures was given in the first part \cite{PCS1} of the
series, where we also showed that any such structure gives rise to a
canonical connection on the tangent bundle of the underlying
manifold. Hence the torsion and the curvature of this canonical
connection are natural invariants of a PACS--structure. The torsion
naturally splits into two components, one of which is exactly the
obstruction to the underlying structure being conformally
symplectic. If this obstruction vanishes, the structure is called a
PCS--structure, and if the remaining component of torsion also vanishes,
one talks about a torsion--free PCS--structure.

Using the classification of simple Lie algebras, one can give an
explicit description of the PACS--structures. In \cite{PCS1} it was
shown that, on the one hand, these structures provide natural
extensions of several well known and interesting geometries. For
example, for any K\"ahler metric, the K\"ahler form and the complex
structure define a torsion--free PCS--structure corresponding to the
simple Lie algebra $\frak{su}(n+1,1)$. Indeed, torsion--free
PCS--structures of this type are equivalent to K\"ahler metrics.
Allowing torsion for the PCS--structure, one obtains certain more
general types of almost Hermitian manifolds, while for
PACS--structures of that type, there is no obvious description in
terms of Hermitian metrics. Things work similarly in indefinite
signatures and for para--Hermitian metrics. Another type of
PCS--structure is closely related to almost quaternionic manifolds
endowed with a conformally symplectic structure that is Hermitian in
the quaternionic sense. 

On the other hand, there is a close relation between PCS--structures
and special symplectic connections in the sense of
\cite{Cahen-Schwachhoefer}. Indeed any special symplectic connection
turns out to be the canonical connection of a torsion--free
PCS--structure, so in particular, this applies to all connections of
exotic symplectic holonomy. There is a nice characterization of the
PCS--structures whose canonical connection is special symplectic using
local parabolic contactifications (see below), which is crucial for
the developments in this article. 

\smallskip

The algebraic data which determine a type of PACS--structures at the
same time determine another geometric structure in one higher
dimension. Any of these structures comes with an underlying contact
structure and they are called \textit{parabolic contact structures},
see the discussion in Section 4.2 of \cite{book}. Now on a contact
manifold, the Reeb field of any contact form defines a
\textit{transversal infinitesimal automorphism} of the contact
structure. In particular it gives rise to a one--dimensional foliation
and any local space of leaves for this foliation naturally inherits a
conformally symplectic structure. As discussed in \cite{Cap-Salac},
any conformally symplectic structure can be locally realized in this
way (``local contactification'') and this realization is unique up to
local contactomorphism. 

For parabolic contact structures, transversal infinitesimal
automorphisms are much more rare (and don't exist generically). Still, 
as shown in the second part \cite{PCS2} of this series there is a
perfect analog of these constructions in the setting of parabolic
contact structures and PCS--structures. For any transversal
infinitesimal automorphism of a parabolic contact structure of any
type, a local leaf space naturally inherits a PCS--structure of the
corresponding type (``PCS--quotients''). Locally, any PCS--structure
can be realized in this way (``parabolic contactification'') and this
realization is unique up to local isomorphism (of parabolic contact
structures). So one can view PCS--structures as geometric structures
characterizing reductions of parabolic contact structures by a
transversal infinitesimal symmetry. 

In the language of parabolic contactifications, one can also deal with
geometries corresponding to Lie algebras of type $C_n$ (which are
excluded in \cite{PCS1}), see Section 3 of \cite{PCS2}. Here the
parabolic contact structure is a \textit{contact projective structure}
(see \cite{Fox}) with vanishing contact torsion, while the analog of a
PCS--structure is a conformally Fedosov structure as introduced in
\cite{E-S} (with slight modifications). Finally, in all cases, the
PCS--structures for which the distinguished connection is special
symplectic (which in the conformally Fedosov case means that it is of
Ricci--type) are exactly those, for which any local parabolic
contactification is locally isomorphic to the homogeneous model. This
gives a conceptual explanation for the construction for special
symplectic connections found in \cite{Cahen-Schwachhoefer}.

\smallskip

As stated above, this last part of the series aims at constructing
differential complexes, which are naturally associated to special
symplectic connections or more general PCS--structures. The original
motivation for the series were the differential complexes on $\Bbb
CP^n$ constructed in \cite{E-G} and applied to problems in integral
geometry there. In that construction, it was not clear what kind of 
geometric structure on $\Bbb CP^n$ is ``responsible'' for the
existence of the complexes. A surprising feature is that these
complexes are one step longer than the de--Rham complex, which
suggests that they have their origin in one higher dimension. The
simplest instance of such a complex is the so--called co--effective
complex on a conformally symplectic manifold, which looks like the
Rumin complex associated to a contact structure in one higher
dimension. Indeed, as a ``proof of concept'' for the current series,
it was shown in \cite{Cap-Salac} that the co--effective complex can be
constructed from the Rumin complex on local contactifications. 

Basically, we carry out a similar procedure in this article, starting
from a large family of differential complexes that are naturally
associated to parabolic contact structures. These are derived from
BGG sequences as introduced in \cite{CSS-BGG}. Standard BGG sequences
are complexes only on locally flat geometries, so pushing them down,
one obtains sequences of differential operators naturally associated
to a PCS--structure, which are complexes provided that the canonical
connection of the PCS--structure is special symplectic. For certain
geometries, it has been shown in \cite{subcomplexes} that certain
parts of BGG sequences are subcomplexes under weaker assumptions than
local flatness. For parabolic contact structures, this only applies in
the case of structures of type $A_n$, so this gives rise to a
construction of complexes (of unusual length) for certain
PCS--structures of K\"ahler and para--K\"ahler type. In the
para--K\"ahler case, one can also start from the relative version of
BGG sequences which were constructed in the recent article
\cite{Rel-BGG2}, and which are complexes under much weaker assumptions
than local flatness.

In the situation of the co--effective complex and the Rumin complex,
both the construction of the upstairs complex and the procedure of
pushing down can be phrased in terms of differential forms. In the
general situation of BGG sequences and their variants, one has to deal
with differential forms with values in a tractor bundle on the level
of parabolic contact structures, and constructing the BGG sequence is
much more involved. Therefore, we use a slightly different approach
than in \cite{Cap-Salac}. The main observation here is that for a
completely reducible natural bundle on a parabolic contact structure
(i.e.~a bundle induced by a completely reducible representations of
the parabolic subgroup) there is an obvious counterpart for
PCS--structure of the corresponding type. For a local
contactification, there is a rather simple relation between sections
upstairs and downstairs, which allows one to directly descend
invariant differential operators acting between sections of such
bundles. Hence we can directly descend the operators in the BGG
sequence to any PCS--quotient, without the need to think about
descending tractor bundles or tractor connections. It should be
remarked that also an approach via downstairs tractor bundles should
be feasible. For the case of conformally Fedosov structures, this has
been carried out in the second version of the preprint \cite{E-S} that
has appeared recently.

After a short review of the geometric structure involved and the
parabolic version of contactification, the push down procedure for
invariant operators acting between sections of completely reducible
bundles is described in Section 2; the main results are Theorems
\ref{thm2.4} and \ref{thm2.5}. Section 3 discusses the applications of this
technique to BGG sequences and the related constructions described
above. We describe in detail the complexes associated to connections
of Ricci type in Theorem \ref{thm3.2} and those associated to
Bochner--bi--Lagrangean metrics in Theorem \ref{thm3.3}. The complexes
for para--K\"ahler manifolds obtained from relative BGG sequences are
described in detail in Theorem \ref{thm3.6}. The cases of complexes for
Bochner--K\"ahler metrics (of any signature) coming from BGG sequences
and for K\"ahler metrics coming from subcomplexes in BGG sequences are
briefly outlined in Section \ref{3.4} and Remark \ref{rem3.6}.

In Section 4, we describe results on the cohomology of the descended
version of BGG sequences. This is similar to the results for the
co--effective complex in \cite{Cap-Salac}, but this time, the main
work is done on the level of the parabolic contact structure. The
basic ingredient here is that on that level, the cohomology of a BGG
sequence can be described as a twisted de--Rham cohomology. A detailed
analysis of the construction of BGG sequences shows that there is a
sequence of subsheaves in the upstairs sheaves of
tractor--bundle--valued differential forms which computes the
cohomology of the descended complex. In Theorem \ref{thm4.4} we construct
a long exact sequence involving the cohomology groups of that
sheaf. Specializing to the case of the homogeneous model, Theorem
\ref{thm4.5} then allows one to interpret this sequence in terms of
``downstairs'' data. The results are analyzed locally as well as for
the global contactification of $\Bbb CP^n$ by the sphere $S^{2n+1}$,
where we obtain a vast generalization of the results on cohomology
needed for the applications in \cite{E-G}.

\section{Pushing down invariant operators}\label{2}
We first review PCS--structures and their relation to parabolic
contact structures. Then we show that each invariant differential
operator acting between sections of irreducible natural bundles on the
parabolic contact structure descends to a natural differential
operator on the corresponding PCS--structure.

\subsection{The types of geometric structures}\label{2.1}
To specify a type of PCS--structure and corresponding parabolic
contact structure, we have to choose some algebraic data, see Section
2.1 of \cite{PCS2} for more details. We first need a semisimple
Lie group $G$, whose Lie algebra $\frak g$ admits a so--called contact
grading $\frak g=\frak g_{-2}\oplus\frak g_{-1}\oplus\frak
g_0\oplus\frak g_1\oplus\frak g_2$. Next, we have to choose a
parabolic subgroup $P\subset G$ corresponding to the Lie subalgebra
$\frak p:=\frak g_0\oplus\frak g_1\oplus\frak g_2$. Then we define a
closed subgroup $G_0\subset P$ with Lie algebra $\frak g_0$ as
consisting of those elements of $P$ whose adjoint action preserves the
grading of $\frak g$. It turns out that the exponential mapping
restricts to a diffeomorphism from $\frak p_+:=\frak g_1\oplus\frak
g_2$ onto a closed normal subgroup $P_+\subset P$ such that $P$ is the
semi--direct product of $G_0$ and $P_+$. In particular, $P/P_+\cong
G_0$. 

By definition, $\frak g_{-}:=\frak g_{-2}\oplus\frak g_{-1}$ is a
Heisenberg algebra, so its Lie bracket defines a non--degenerate line
in $\La^2(\frak g_{-1})^*$. This in turn defines the conformally
symplectic group $CSp(\frak g_{-1})\subset GL(\frak g_{-1})$. It turns
out that for any element $\ph\in CSp(\frak g_{-1})$, there is a unique
linear isomorphism $\ps:\frak g_{-2}\to\frak g_{-2}$ such that
$(\ps,\ph)$ defines an automorphism of the graded Lie algebra $\frak
g_{-}$, so one obtains an isomorphism $\Aut_{gr}(\frak g_-)\cong
CSp(\frak g_{-1})$. By definition, the adjoint action of $G_0$
restricts to an action by automorphisms on the graded Lie algebra
$\frak g_-$, so one obtains a homomorphism $G_0\to CSp(\frak g_{-1})$
which turns out to be infinitesimally injective. 

If $\frak g$ is not of type $C_n$, then both geometric structures we
need are determined by this homomorphism. In \cite{PCS1}, we have
defined the \textit{PACS--structure} associated to $(G,P)$ as the
first order structure on manifolds of dimension $\dim(\frak g_{-1})$
determined by the homomorphism $G_0\to GL(\frak g_{-1})$. Hence such a
structure on a smooth manifold $M$ is given by a principal
$G_0$--bundle together with a \textit{soldering form}, a strictly
horizontal, $G_0$--equivariant $\frak g_{-1}$--valued one--form on the
total space of this bundle. Since the image of our homomorphism is
contained in $CSp(\frak g_{-1})$, any such structure induces an
underlying almost conformally symplectic structure. A
\textit{PCS--structure of type $(G,P)$} is then such a first order
structure for which this underlying structure is conformally
symplectic.

On the other hand, for a contact manifold of dimension $\dim(\frak
g_-)$, one considers the associated graded to the tangent bundle,
which has a natural frame bundle with structure group $\Aut_{gr}(\frak
g_-)\cong CSp(\frak g_{-1})$. A \textit{parabolic contact structure of
  type $(G,P)$} is then given by a reduction of structure group of
this frame bundles corresponding to the homomorphism $G_0\to CSp(\frak
g_{-1})$, respectively by the canonical Cartan geometry that such a
reduction determines, see Section 4.2 of \cite{book}.

If $\frak g$ is of type $C_n$, then it turns out that the homomorphism
$G_0\to CSp(\frak g_{-1})$ induces an isomorphism between the Lie
algebras of the two groups. Thus reductions of structure group as
considered above carry very little information. Nonetheless, there are
analogs for both types of geometries in the $C_n$--case. On the
parabolic contact side, these are \textit{contact projective
  structures} (with vanishing contact torsion) as discussed in
\cite{Fox} and in Section 4.2.6 of \cite{book}. On the conformally
symplectic side, these are the \textit{conformally Fedosov structures}
discussed in Section 3 of \cite{PCS2} based on the earlier treatment
in the first version of \cite{E-S}.

Treating the structures in terms of principal bundles and soldering
forms (of appropriate type) the $C_n$--case looks essentially the same
as the other cases. Hence we will give a uniform treatment below and
also refer to conformally Fedosov structures as PCS--structures of
type $C_n$.

\subsection{Invariant operators on parabolic contact
  structures}\label{2.2} 

The uniform description of parabolic contact structures is via Cartan
geometries of type $(G,P)$. A parabolic contact structure of type
$(G,P)$ on a manifold $M^\#$ determines a principal $P$--bundle
$p^\#:\Cal G^\#\to M^\#$ and a normal Cartan connection
$\om\in\Om^1(\Cal G^\#,\frak g)$. Factoring by the free action of
$P_+\subset P$, we obtain a principal $G_0$--bundle $\Cal G_0^\#\to
M^\#$. The Cartan connection $\om$ descends to a soldering form
$\th^\#$ that defines a reduction of the frame bundles of
the associated graded to the tangent bundle, see Section 2.3 of
\cite{PCS2} for more details. If $\frak g$ is not of type $C_n$, then
this is the equivalent encoding of the geometry as discussed in
Section \ref{2.1}.

Any representation of the group $P$ gives rise to a natural vector
bundle on parabolic contact structures of type $(G,P)$ via forming
associated bundles to the Cartan bundle. We will only meet natural
bundles obtained in this way in this article. The general
representation theory of $P$ is rather complicated, but irreducible
and hence completely reducible representations of $P$ are easy to
understand. If $\Bbb W$ is such a representation, then the nilpotent
normal subgroup $P_+\subset P$ acts trivially on $\Bbb W$, so we
obtain a representation of $G_0$. This
immediately implies that the associated bundle $\Cal G^\#\x_P\Bbb W$
can be naturally identified with $\Cal G_0^\#\x_{G_0}\Bbb W$. Hence
natural bundles associated to completely reducible representations can
be readily understood in terms of the underlying structure. Finally,
the group $G_0$ is always reductive, so its representation theory is
well understood.

The equivalence between parabolic contact structures and Cartan
geometries in particular implies that any automorphism of a parabolic
contact structure on $M^\#$ lifts to a bundle--automorphism of $\Cal
G^\#$ which preserves $\om$. This implies an analogous result for an
infinitesimal automorphism $\xi\in\frak X(M^\#)$, i.e.~a vector field
whose local flows are automorphisms. Such a vector field always
uniquely lifts to a $P$--invariant vector field $\tilde\xi\in\frak
X(\Cal G^\#)$ such that for the Lie derivative $\Cal L$, we get $\Cal
L_{\tilde\xi}\om=0$. Invariance of $\tilde\xi$ implies that there is
an intermediate vector field $\xi_0\in\frak X(\Cal G^\#_0)$ whose flow
preserves the soldering form. We will be particularly interested in
the case of \textit{transverse} infinitesimal automorphisms, i.e.~the
case that all values of $\xi\in\frak X(M^\#)$ are transverse to the
contact subbundle (so in particular, $\xi$ is nowhere vanishing). 

Given a representation $\Bbb W$ of $P$, the space of sections of the
natural bundle $\Cal G^\#\x_P\Bbb W$ can be naturally identified with
the space of $P$--equivariant smooth functions $\Cal G^\#\to\Bbb
W$. Given vector field $\tilde\xi\in\frak X(\Cal G^\#)$, we can
differentiate such equivariant functions and if $\tilde\xi$ is
$P$--invariant, then the resulting function will be equivariant and
hence correspond to a section. Hence an equivariant vector field acts
on sections of any natural vector bundle, and we will denote this
action by $\Cal L_{\tilde\xi}$.

If $\Bbb W$ is completely reducible then $P$--equivariancy of a
function $\Cal G^\#\to\Bbb W$ implies invariance under the group
$P_+$. Hence such a function descends to $\Cal G^\#/P_+=\Cal G_0^\#$
and is $G_0$--equivariant there. A $P$--invariant vector field
$\tilde\xi$ as above induces a $G_0$--invariant vector field $\xi_0$
on $\Cal G_0^\#$ and we can use $\xi_0$ to differentiate sections as
above, thus obtaining the same action as above. 

There is a general concept of invariant differential operators acting
between sections of natural vector bundles over manifolds endowed with
a parabolic contact structure of some fixed type. For our puposes it
suffices to know that such an operator is defined on any manifold
endowed with a structure of the given type and that these operators
are compatible with the inclusion of open subsets (endowed with the
restricted structure) and with the action of isomorphisms. Hence they
are compatible with the action of local isomorphism and in particular
of local automorphisms. Applying this to local flows, we conclude that
for an infinitesimal automorphism $\tilde\xi$ of the Cartan geometry
determined by a parabolic contact structure, any invariant
differential operator $D$ commutes with the action of $\Cal
L_{\tilde\xi}$. In the case of completely reducible bundles, one may
as well work on the $G_0$--principal bundle $\Cal G^\#_0$ using $\Cal
L_{\xi_0}$.

\subsection{PCS--quotients}\label{2.3}

The fundamental notion for the study of contactifications in the realm
of PCS--structures and parabolic contact structures in \cite{PCS2} is
a PCS--quotient. Suppose that we have given a parabolic contact
structure of type $(G,P)$ on $M^\#$ together with a  transverse infinitesimal
automorphism $\xi\in\frak X(M^\#)$ of this geometry. Then $\xi$ is
nowhere vanishing and hence defines a one--dimensional foliation of
$M^\#$.  Basically, a PCS--quotient is a global space of leafs for
this foliation which is endowed with a PCS--structure of type $(G,P)$
that can be viewed as a quotient of the parabolic contact structure. 

To formulate the precise definition, let $p_0^\#:\Cal G_0^\#\to M^\#$
be the $G_0$--bundle determined by the parabolic contact structure and
let $\th^\#$ be its soldering form. Now for a PCS--quotient, one
requires 
\begin{itemize}
\item a surjective submersion $q:M^\#\to M$ onto a smooth manifold $M$
  such that the fibers of $q$ are connected and their tangent spaces
  are spanned by $\xi$
\item a PCS--structure of type $(G,P)$ on $M$ with $G_0$--bundle
  $p:\Cal G_0\to M$ and soldering form $\th$
\item a lift $q_0:\Cal G_0^\#\to\Cal G_0$ of $q$ to a morphism of
  principal bundles which is a surjective submersion with connected
  fibers, whose tangent spaces are spanned by $\xi_0$ and such that
  $q_0^*\th$ coincides with the component $\th^\#_{-1}$ of the
  ``upstairs'' soldering form (see Sections 2.3 and 2.4 of \cite{PCS2}
  for details)
\end{itemize}

\begin{remark}\label{rem2.3}
If $\frak g$ is not of type $C_n$, then the above is exactly the
definition of a PCS--quotient from Section 2.4 of \cite{PCS2}. If
$\frak g$ is of type $C_n$, then the discussion in Section 3.3 of
\cite{PCS2} shows that the same setup is available for a projective
contact structure (with vanishing contact torsion) on $M^\#$ and a
conformally Fedosov structure on $M$, see in particular the proof of
the first part of Theorem 8 of \cite{PCS2}.
\end{remark}

As it stands, the concept of a PCS--quotient may look rather
restrictive and one might doubt whether there are many
examples. However, the results of \cite{Cap-Salac} and of \cite{PCS2}
imply that there are lots of examples. This is best formulated in the
language of \textit{parabolic contactifications} of
PCS--structures. By a parabolic contactifications of a PCS--structure
$M$, we simply mean a realization of $M$ as a PCS--quotient of a
parabolic contact structure of type $(G,P)$ (respectively that
parabolic contact structure). For later use, let us collect the
fundamental results on parabolic contactifications:

\begin{thm}\label{thm2.3}
(1) Let $M$ be a PCS--structure with underlying conformally symplectic
  structure $\ell\subset\La^2T^*M$. Then any open subset $U\subset M$
  over which $\ell$ admits a nowhere--vanishing section which is exact
  as a two--form on $M$ admits a parabolic contactification.

  (2) Let $M$ and $N$ be two PCS--structures endowed with fixed
  parabolic contactifications. Then locally any morphism of
  PCS--structures (compare with Section \ref{2.5} below) lifts to a
  contactomorphism between the contactifications. 

(3) Any lift of a morphism of PCS--structures to a contactomorphism of
  parabolic contactifications is automatically compatible with the
  infinitesimal automorphisms up to a nowhere--vanishing, locally
  constant factor and a morphism of parabolic contact structures.
\end{thm}
\begin{proof}
  (1) By Lemma 3.1 of \cite{Cap-Salac}, $U$ can be realized as the
  quotient $q:U^\#\to U$ of a contact manifold $U^\#$ by a transverse
  infinitesimal contactomorphism. By Theorem 4 (for $\frak g$ not of
  type $C_n$)respectively part 2 of Theorem 8 (for $\frak g$ of type
  $C_n$) of \cite{PCS2}, a PCS--structure of type $(G,P)$ on $U$ lifts
  to a parabolic contact structure of type $(G,P)$ on $U^\#$, thus
  providing the required PCS--quotient.

\smallskip

The statement of (3) is proved in Theorem 5 (for $\frak g$ not of type
$C_n$) respectively in part 3 of Theorem 8 (for $\frak g$ of type
$C_n$) of \cite{PCS2}. Moreover, by Proposition 3.1 of
\cite{Cap-Salac}, the assumption of part (3) is locally satisfied for
each morphism between PCS--structures endowed with contactifications,
so (2) follows.
\end{proof}

We want to remark that in \cite{PCS2} we also constructed examples of
global contactifications which will play an important role later on. 

\subsection{Descending invariant differential operators to
  PCS--quotients}\label{2.4} 
Now suppose that $\Bbb W$ is a representation of $G_0$, which we can
also view as a completely reducible representation of $P$. Then as
discussed in Section \ref{2.2}, this gives rise to a (completely
reducible) natural vector bundle on parabolic contact structures of
type $(G,P)$. Given such a geometry $(p^\#:\Cal G^\#\to M^\#,\om)$ we
denote the resulting bundle by $\Cal WM^\#:=\Cal G^\#\x_P\Bbb W$. As
noted in Section \ref{2.2}, we can also view $\Cal WM^\#$ as the
bundle $\Cal G_0^\#\x_{G_0}\Bbb W$ associated to the underlying
$G_0$--bundle.

On the other hand, one also obtains a natural vector bundle on
PCS--structures of type $(G,P)$, since they are also defined by a
principal $G_0$--bundle. Given such a geometry $(p:\Cal G_0\to M,\th)$
we write $\Cal WM:=\Cal G_0\x_{G_0}\Bbb W$ for this bundle.

Now it is well known that sections of an associated bundle can be
viewed as equivariant functions on the total space of the inducing
principal bundle. Explicitly, the space $\Ga(\Cal W\to M)$ of sections
is naturally isomorphic to 
$$
C^\infty(\Cal G_0,\Bbb W)=\{f\in C^\infty(\Cal G_0,\Bbb W):f(u\cdot
g)=g^{-1}\cdot f(u)\quad\forall g\in G_0\}. 
$$
Evidently, such a function can be pulled back via $q_0:\Cal
G_0^\#\to\Cal G_0$ to a smooth equivariant function $\Cal
G_0^\#\to\Bbb W$, which then defines a smooth section of $\Cal
WM^\#\to M^\#$. This defines an injection $\Ga(\Cal
WM)\hookrightarrow\Ga(\Cal WM^\#)$, which we denote by $q_0^*$.

\begin{lemma}\label{lem2.4}
  Let $q:M^\#\to M$ be a PCS--quotient by a transversal infinitesimal
  automorphism $\xi\in\frak X(M^\#)$ of a parabolic contact structure
  of type $(G,P)$ with bundle map $q_0:\Cal G_0^\#\to\Cal G_0$. Let
  $\Bbb W$ be a representations of $G_0$ and consider the
  corresponding induced bundles $\Cal WM^\#$ and $\Cal WM$ as
  above. Let $\xi_0\in\frak X(\Cal G_0^\#)$ be the $G_0$--invariant
  vector field induced by $\xi$ and consider the induced map $\Cal
  L_{\xi_0}$ on the space $\Ga(\Cal WM^\#)$.

Then the the image of $q_0^*:\Ga(\Cal WM)\to\Ga(\Cal WM^\#)$ coincides
with the kernel of $\Cal L_{\xi_0}$. 
\end{lemma}
\begin{proof}
  In the language of equivariant functions, $\Cal L_{\xi_0}$ is simply
  given by differentiating vector valued functions using the vector
  field $\xi_0$. (This preserves the space of equivariant functions
  since $\xi_0$ is $G_0$--invariant.) In view of this, the result
  follows from the description of $q_0^*$ in terms of equivariant
  functions, since the fibers of $q_0$ are connected by assumption and
  their tangent spaces are spanned by $\xi_0$.
\end{proof}

Having this at hand, we can formulate the fundamental results about
descending invariant differential operators.

\begin{thm}\label{thm2.4}
  Let $q:M^\#\to M$ be a PCS--quotient by a transversal infinitesimal
  automorphism $\xi\in\frak X(M^\#)$ of a parabolic contact structure
  of type $(G,P)$ with bundle map $q_0:\Cal G_0^\#\to\Cal G_0$, and
  let $\Bbb W$ and $\tbw$ be representations of $G_0$. Further, let
  $D^\#:\Ga(\Cal WM^\#)\to\Ga(\tcw M^\#)$ be a linear invariant
  differential operator for the given parabolic contact structure.

  Then there is a unique linear differential operator
  $D:\Ga(\Cal WM)\to\Ga(\tcw M)$ such that $q_0^*\o D=D^\#\o q_0^*$.
\end{thm}
\begin{proof}
  Since $D^\#$ is an invariant differential operator, it commutes with
  pullback along the flow of $\xi_0$. Infinitesimally, this means that
  $D^\#\o\Cal L_{\xi_0}=\Cal L_{\xi_0}\o D^\#$, so in particular
  $D^\#$ maps $\ker(\Cal L_{\xi_0})\subset\Ga(\Cal WM^\#)$ to $
  \ker(\Cal L_{\xi_0})\subset\Ga(\tcw M^\#)$. Using Lemma \ref{lem2.4},
  we conclude that given $\si\in\Ga(\Cal WM)$, there is a unique
  section $\tilde\si$ such that $D^\#(q_0^*\si)=q_0^*\tilde\si$, so we
  can define $D(\si):=\tilde\si$ to obtain an operator with the
  desired property. Clearly, $D$ is linear, and in local coordinates
  it is evident that $D$ is a differential operator. Alternatively,
  one may observe that if $\si$ vanishes on an open subset $U\subset
  M$, then $q_0^*\si$ vanishes on $(q_0)^{-1}(U)$. Since $D^\#$ is a
  differential operator, the same holds for $D^\#(q_0^*\si)$ so
  $D(\si)$ vanishes on $U$. Thus $D$ is a local operator and thus a
  differential operator by the Peetre theorem.
\end{proof}

\subsection{Naturality of the descended operators}\label{2.5} 
We next show that pushing down to PCS--quotients can be used to
construct natural operators on the category of PCS--structures from
invariant differential operators for the corresponding parabolic
contact structure. To explain the meaning of ``natural operator'', we
have to recall some concepts.

First a morphism of PCS--structures is defined to be a principal
bundle morphism $\ph$ which covers a local diffeomorphism
$\underline{\ph}$ of the base manifolds and is compatible with the
soldering forms. Next, we need the concept of a natural vector bundle
on the category of PCS--structures, but here we restrict to bundles
associated to the defining principal bundle. So as in Section
\ref{2.4}, we take a representation $\Bbb W$ of $G_0$ and for a
PCS--structure $(\Cal G_0\to M,\th)$ we define $\Cal WM:=\Cal
G_0\x_{G_0}\Bbb W$. The soldering form $\th$ can then be used to
identify natural bundles of this type with more traditional natural
bundles like tensor bundles.

This implies that any morphism $\ph$ of PCS--structures, say on $M$
and $N$, induces a vector bundle map $\Cal W\ph:\Cal WM\to\Cal WN$,
which restricts to a linear isomorphism on each fiber. Compatibility
with the soldering forms implies that this is compatible with the
identifications with tensor bundles, i.e.~one obtains the usual induced
bundle maps there. The induced vector bundle maps can then be used to
pull back sections of associated bundles: For $\si\in\Ga(\Cal WN)$,
there is a unique section $\ph^*\si\in\Ga(\Cal WM)$ such that $\si\o
\underline{\ph}=\Cal W\ph\o \ph^*\si$. 

Now given a second representation $\tbw$, a natural operator between
sections of the corresponding associated bundles is defined as a
family of differential operators $D_M:\Ga(\Cal WM)\to\Ga(\tcw M)$
which is compatible with the actions of all pullback operators
associated to morphisms of PCS--structures. Hence for any morphism
$\ph$ to a PCS--structure over $N$, and any section $\si\in\Ga(\Cal
WN)$ we require $D_M(\ph^*\si)=\ph^*(D_N(\si))$, where (as usual) we
denote all pullback operators by the same symbol. For the moment, we
stick to this general concept, some remarks on more restrictive
concepts of invariant operators are made below. 

\begin{thm}\label{thm2.5}
Let $\Bbb W$ and $\tbw$ be two representations of $G_0$, which we also
view as completely reducible representation of $P$. Then any invariant
operator on the category of parabolic contact structures of type
$(G,P)$ between sections of the natural bundles induced by the two
representations canonically induces a natural differential operator on
the category of PCS--structures of type $(G,P)$ acting between
sections of the induced bundles corresponding to the two
representations. 
\end{thm}
\begin{proof}
This follows rather easily from the results on PCS--contactifications
in Theorem \ref{thm2.3}. Let us start with a PCS--structure on $M$, a
section $\si\in\Ga(\Cal WM)$ and a point $x\in M$. By part (1) of
Theorem \ref{thm2.3}, there is an open neighborhood $U$ of $x$ in $M$
which can be realized as a PCS--quotient $q:U^\#\to U$. Given the
invariant operator $D^\#$ on parabolic contact structures, we can use
Theorem \ref{thm2.4} to obtain an operator $D_U:\Ga(\Cal WU)\to \Ga(\tcw
U)$. In particular, we can apply this to $\si|_U$ to obtain a section
of $\tcw M$ defined over $U$.

To complete the proof, we need a fact on pullbacks of sections. (This
may look rather obvious in written form, but this is slightly
deceiving, since this relates two different concepts of pullback,
which are denoted in the same way. In particular, one of this is
non--standard since it relates bundles over different manifolds.)
Suppose that $M$ and $\tilde M$ are PCS--structures endowed with
parabolic contactifications $q:M^\#\to M$ and $\tilde q:\tilde
M^\#\to\tilde M$, and let us denote the corresponding bundles by $\Cal
G_0$, $\tcg_0$, $\Cal G_0^\#$ and $\tcg_0^\#$, respectively. Assume
further that $\Ph:\Cal G_0\to\tcg_0$ is a morphism of PCS--structures
and that $\Ps:\Cal G_0^\#\to\tcg_0^\#$ is a lift to a morphism of
parabolic contact structures which is compatible with the
infinitesimal automorphisms up to a constant multiple. Then for any
section $\si$ of a natural bundle over $\tilde M$, we have
$\Ps^*\tilde q^*_0\si=q^*_0\Ph^*\si$.

To prove this claim, observe that if $f$ is the equivariant function on
$\tcg_0$ corresponding to $\si$, then $\tilde q^*_0\si$ and $\Ps^*\tilde
q^*_0\si$ correspond to $f\o\tilde q_0$ and $f\o\tilde q_0\o\Ps$,
respectively. Since $\Ps$ is compatible with the infinitesimal
automorphisms up to constant multiple, the fact that $\tilde q^*_0\si$
lies in the kernel of $\Cal L_{\tilde\xi_0}$ implies that $\Ps^*\tilde
q^*_0\si$ lies in the kernel of $\Cal L_{\xi_0}$.  Thus it must be of the
form $q^*_0\tau$ for some section $\tau$ and then $\tilde
q_0\o\Ps=\Ph\o q_0$ implies that $\tau=q^*_0\Ph^*\si$, which completes
the proof of the claim. 

Returning to $\Cal WM$ and $\tcw M$, we can carry out the above
construction for the elements of an open covering $\{U_i:i\in I\}$ of
$M$, so for each $i$ we descend $D^\#(q_i^*(\si|_{U_i}))$ to a section
$D_i(\si)$ of $\tcw M$ defined over $U_i$. If $U_i\cap
U_j=U_{ij}\neq\emptyset$, then by parts (2) and (3) of Theorem
\ref{thm2.3}, the identity on $U_{ij}$ locally lifts to an isomorphism
$\Ps_{ij}$ of the parabolic contact structures which is compatible
with the transversal infinitesimal automorphisms up to a constant
multiple. By the claim, this implies that over the open subset in
question, we have $\Ps_{ij}^*q_j^*\si=q_i^*\si$. Invariance of $D^\#$
now implies that $D^\#(q_i^*\si)=\Ps_{ij}^*D^\#(q_j^*\si)$, so
$q_i^*(D_i(\si))=\Ps_{ij}^*q_j^*(D_j(\si))$. Again by the claim, the
right hand side equals $q_i^*(D_j(\si))$. Thus, locally on $U_{ij}$,
we have $D_i(\si)=D_j(\si)$, so this has to hold on all of $U_{ij}$.

On the one hand, this shows that the sections $D_i(\si)$ can be pieced
together to define a global section $D_M(\si)$. On the other hand,
applying the argument to the union of two coverings, we see that
$D_M(\si)$ is independent of the choice of covering, so we have
obtained a well defined linear differential operator $D_M$. Thus it
remains to prove that the $D_M$ define a natural operator. Naturality
of a differential operator can be verified locally, so in view of
Theorem \ref{thm2.3}, it suffices to do this for PCS--structures
admitting a global contactification and for morphisms which lift to
the contactifications. But in this case, the required property follows
immediately from the claim.
\end{proof}

\section{(Relative) BGG complexes and subcomplexes}\label{3} 
Since PCS--structures admit canonical connections, constructing
differential operators, which are intrinsic to such structures, is not
difficult. As in Riemannian geometry, one can simply form iterated
covariant derivatives with respect to induced linear connections,
combine them with iterated covariant derivatives of the torsion and
the curvature of the canonical connection and then apply tensorial
operations. Constructing differential \textit{complexes} naturally
associated to such geometries is a completely different issue, and it
is not at all clear, how to do this ``by hand''.

On the other hand, there are general constructions for a large number
of differential complexes on locally flat parabolic contact structures
as well as on certain non--flat structure of type $A_n$. All these
complexes can be pushed down to PCS--quotients thus providing a large
number of differential complexes, which are naturally associated to
such structures.

\subsection{BGG sequences}\label{3.1} 
Let us start with a type $(G,P)$ of parabolic geometries and a finite
dimensional representation $\Bbb V$ of $G$. Associated to these data
there is a sequence of invariant differential operators acting on
sections of certain irreducible natural vector bundles over parabolic
geometries of type $(G,P)$. A construction for these sequences was
given in \cite{CSS-BGG} and improved in \cite{Calderbank-Diemer}. More
recently, the construction was generalized and substantially improved
in \cite{Rel-BGG2}. In the case of the homogeneous model $G/P$ of the
geometry, the resulting sequence turns out to be a complex and a fine
resolution of the locally constant sheaf $\Bbb V$. In a certain sense
this resolution is dual to Lepowsky's generalization (see
\cite{Lepowsky}) of the Bernstein--Gelfand--Gelfand resolution of
$\Bbb V$ by homomorphisms of Verma modules. The fact that the BGG
sequence is a complex, extends from the homogeneous model $G/P$ to all
parabolic contact structures which are locally isomorphic to $G/P$,
i.e.~to the locally flat geometries. Also in this more general case,
the BGG complex is a fine resolution of a sheaf, which can be described
explicitly as the sheaf of those sections of the tractor bundle
associated to $\Bbb V$, which are parallel for the canonical tractor
connection. 

Applying the push--down construction from Section \ref{2} to a
BGG-sequence, one therefore obtains a complex if all local
contactifications of a given PCS--structure are locally flat. The
latter property is analyzed in Theorem 7 and Corollary 1 of
\cite{PCS2}, where it is shown to be equivalent to the fact that the
canonical connection associated to the PCS--structure is a special
symplectic connection in the sense of
\cite{Cahen-Schwachhoefer}. Since conversely any special symplectic
connection is the canonical connection associated to a PCS--structure,
we conclude that the pushed down versions of BGG complexes are
associated to special symplectic connections. We should point out here
that global contactifications as discussed in Sections 2.6 and 3.4 of
\cite{PCS2} are of particular interest here. These are
contactifications of compact PCS--structures, which are circle
bundles, and in this case it is possible to analyze the cohomology of
the resulting complexes, see Section \ref{4}.

The bundles showing up in a BGG sequence are associated to the
representations of $P$ on the Lie algebra homology groups $H_*(\frak
p_+,\Bbb V)$. Here $\frak p_+$ is the nilradical of the parabolic
subalgebra $\frak p\subset\frak g$, and there is a general result that
these homology representations are always completely reducible. Hence
in the complex case, they can be described in terms of weights and
Kostant's theorem (see \cite{Kostant}) gives an explicit way to
compute the relevant weights algorithmically. These results can be
extended to the real case using complexifications. In what follows, we
will usually suppress such computations and just describe the
resulting bundles explicitly.

Invariance of the operators in a BGG sequence can also be used to
determine the principal part of the operator. By construction, the
principal symbol of any operator showing up in a BGG sequence has to
be a natural bundle map and thus is induced by a $P$--homomorphism 
between the inducing representations. As noted above, the inducing
representations are completely reducible, so the action of $P$ comes
from a representation of the reductive group $G_0$, whose
representation theory is well understood. 

To described the BGG complexes associated to special symplectic
connections, the main task therefore is to convert the representation
theory information available for the parabolic contact structures into
information on bundles on PCS--quotients. We will do this in a bit
more detail for the $C_n$ and $A_n$ types and sketch how things look
for the other types. 

\subsection{Example: Complexes associated to connections of Ricci
  type}\label{3.2}  

This is the case discussed in Section 3 of \cite{PCS2}. One starts
with a conformally Fedosov structure on a smooth manifold $M$, which
is given by a conformally symplectic structure $\ell\subset \La^2T^*M$
and a projective class of torsion free linear connections on $TM$,
which satisfy a certain compatibility condition. Locally, this
structure determines a symplectic form $\om$ on $M$ (up to a constant
multiple) and a unique connection $\nabla$ in the projective class
such that $\nabla\om=0$. So locally, the structure is just given by a
torsion--free symplectic connection, see Proposition 2 in
\cite{PCS2}. Any local contactification of $M$ then inherits a
canonical contact projective structure, which is locally flat if and
only if the connection $\nabla$ is of Ricci type, see Theorem 8 and
Corollary 1 in \cite{PCS2}.

The irreducible natural bundles available in this situation are easy
to describe. They are equivalent to irreducible representation of
$\frak g_0\cong\mathfrak{csp}(2n)$, where $2n=\dim(\frak
g_{-1})$. Irreducible representations of the center are
one--dimensional and thus give rise to natural line bundles. On the
other hand the irreducible representations of $\frak{sp}(2n)$ can all
be constructed from the standard representation by tensorial
operations. For a contact projective structure on $M^\#$, the standard
representation of $\frak g_0$ corresponds to the contact subbundle
$H\subset TM^\#$. The basic natural line bundle in this case is the
quotient $Q:=TM^\#/H$, and one can construct density bundles as (real)
roots of the line bundle $Q\otimes Q$, which has to be trivial. 

For a PCS--structure on $M$, the standard representation of $\frak
g_0$ corresponds to the tangent bundle $TM$. Now on $M^\#$, the Levi
bracket induces an isomorphism $H\cong H^*\otimes Q$, whereas on $M$,
inserting vector fields into elements of $\ell$ defines an isomorphism
$TM\otimes\ell\to T^*M$, which shows that the representation $\frak
g_{-2}$ of $\frak g_0$, which gives rise to $Q$ on $M^\#$ corresponds
to $\ell^*$ on $M$, compare with Section 3.2 of \cite{Cap-Salac}. This
is sufficient to explicitly associate to any irreducible
representation of $\frak g_0$ a weighted tensor bundle (a tensor
product of a natural line bundle with a natural subbundle of a tensor
bundle) on $M$.

Using this, we can give a description of the resulting sequences in
the spirit of the parametrization of BGG sequences for AHS--structures
introduced in \cite{BCEG}. To do this, we make one more
observation. Suppose that $V$ and $W$ are two irreducible
representations of $\frak g_0$. Then in the tensor product $V\otimes
W$, there is a specific irreducible component called the
\textit{Cartan product}, which we denote by $V\odot W$. This is the
component of maximal highest weight respectively the subrepresentation
generated by the tensor product of two highest weight vectors. Given
two natural tensor bundles $E$ and $F$, we denote by $E\odot F\subset
E\otimes F$ the irreducible tensor subbundle corresponding to the
Cartan product of the inducing representations. Since $V\odot W$
occurs in $V\otimes W$ with multiplicity one, there is a unique (up to
scale) natural bundle map $E\otimes F\to E\odot F$, which we call the
canonical projection onto the Cartan product.

\begin{thm}\label{thm3.2}
Let $E$ be an irreducible tensor bundle on conformally symplectic
manifolds of dimension $2n\geq 4$ and let $k\geq 0$ be an integer,
which, depending on $E$ has to be even or odd. Then pushing down an
appropriate BGG sequence on local contactifications, one obtains on any
conformally Fedosov manifold $M$ of dimension $2n$ a sequence of
weighted irreducible tensor bundles and invariant differential
operators of the form
$$
\Ga(E_0)\overset{D_0}{\longrightarrow} \Ga(E_1)
\overset{D_1}{\longrightarrow} \cdots
\overset{D_{2n-1}}{\longrightarrow} \Ga(E_{2n}) \overset{D_{2n}}{\longrightarrow} 
\Ga(E_{2n+1}).
$$ This sequence is a complex if the canonical connection $\nabla$ of
$M$ is of Ricci--type. Moreover, $E_0$ is a tensor product of $E$ with
some density bundle, and $E_1$ is the Cartan product $S^{k+1}T^*M\odot
E_0$. The operator $D_0$ has order $k+1$ and its principal part is
given by forming the $k+1$--fold covariant derivative with respect to
(the connection induced by) $\nabla$, symmetrizing and then projecting
to the Cartan product.
\end{thm}
\begin{proof}
We have $\frak g=\mathfrak{sp}(2n+2,\Bbb R)$, $\frak p\subset\frak g$
is the stabilizer of a line in the standard representation. The
semisimple part $\frak g_0^0$ of $\frak g_0$, which is isomorphic to
$\frak {sp}(2n,\Bbb R)$, can be identified with the space of those
maps, which vanish on a non--degenerate plane containing this
line. Since $\frak g$ is a split real form, it has a root
decomposition and there is a simple root $\al_1$, such that a root
space $\frak g_\al$ lies in $\frak g_0^0$ if and only if $\al$ is
linear combination of the other simple roots
$\al_2,\dots,\al_{n+1}$. Now let $\om_1,\dots,\om_{n+1}$ be the
corresponding fundamental weights, so dominant integral weights for
$\frak g$ are linear combinations of these weights with non--negative
integral coefficients. If such a linear combination does not involve
$\om_1$, then it can naturally be viewed as a weight of $\frak g_0^0$,
and all weights of $\frak g_0^0$ arise in this way.

For our purposes, it is better to describe representations by the
negatives of lowest weights rather than using the usual description in
terms of highest weights, but this causes only small differences. The
irreducible tensor bundle $E$ then corresponds to a weight of $\frak
g_0^0$. Representing this weight as a linear combination of
$\om_2,\dots,\om_{n+1}$ and adding $k\om_1$ (where $k$ is the chosen
integer), we obtain a dominant integral weight. This corresponds to a
finite dimensional irreducible representation $\Bbb V$ of $\frak g$,
which integrates to the group $Sp(2n+2,\Bbb R)$. Now let us in
addition assume that the sum of the coefficients of those $\om_i$ with
odd $i$ is even, which, depending on $E$, means that $k$ has to be even
or that $k$ has to be odd. Then the homomorphism $Sp(2n,\Bbb R)\to
GL(\Bbb V)$ defining the representation factorizes to
$G:=PSp(2n+2,\Bbb R)$, and hence $\Bbb V$ gives rise to a
BGG--sequence on parabolic geometries of type $(G,P)$ which are
equivalent to contact projective structures.

Via the construction in Section \ref{2} we can descend this to a
sequence of invariant differential operators on conformally Fedosov
structures. The BGG sequence is a complex if the contact projective
structure is locally flat, so we obtain a complex if the canonical
connection of the conformally Fedosov structure is of Ricci type. To
prove the rest of the theorem we need some information on the bundles
occurring in the BGG sequence, which all follow from the description
of the homology groups $H_*(\frak p_+,\Bbb V)$ via Kostant's
theorem. The homology groups split into a direct sum of different
irreducible representations of $\frak g_0$, and the corresponding
weights are obtained from the weight determined by $\Bbb V$ by the
affine action of a certain subset $W^{\frak p}$ of the Weyl group $W$
of $\frak g$. The homology degree in which an irreducible component
occurs is given by the length of the corresponding Weyl group element.

Using the algorithms from Section 3.2.16 of \cite{book}, one easily
verifies that $W^{\frak p}$ consists of $2n+2$ elements, which have
length $0$, $1$, \dots, $2n+1$, respectively. This implies that the
BGG sequence and hence the descended sequence has the claimed form
with an irreducible bundle in each degree between $0$ and $2n+1$. The
unique element of length zero in $W^{\frak p}$ is the identity, so
$H_0(\frak p_+,\Bbb V)$ is the representation of $\frak g_0$
corresponding to the same weight as $\Bbb V$. But this exactly says
that $E_0$ is the tensor product of $E$ with a natural line
bundle. The unique element of length one in $W^{\frak p}$ is the
simple reflection corresponding to $\al_1$. The affine action by this
reflection maps $k\om_1+a_2\om_2+\sum_{i\geq 3}a_i\om_i$ to
$(-k-2)\om_1+(a_2+k+1)\om_2 +\sum_{i\geq 3}a_i\om_i$, which is just
the sum of the initial weight with $-(k+1)\al_1$. Now $\al_1$ is the
lowest weight of $\frak g_1\cong\frak g_{-1}^*$, so $-(k+1)\al_1$ is
the negative of the lowest weight of the representation inducing
$S^{k+1}T^*M$, which implies the claim on $E_1$. The claim on the
principal part of $D_0$ then follows from invariance on the level of
contact projective structures.
\end{proof}

Via Kostant's theorem, the representations inducing the bundles $E_i$
in the sequence can be determined explicitly and
algorithmically. However, for $i\geq 2$, the explicit form of the
bundles depends on the initial representation in a more complicated
way. Let us just describe one situation in a bit more detail, which in
particular covers the complexes used in \cite{E-G}. To obtain these
complexes, we use the global contactification $S^{2n+1}\to\Bbb CP^n$
defined by the Hopf--fibration. This can be interpreted as a
PCS--contactification of the conformally Fedosov structure on $\Bbb
CP^n$ defined by the Levi--Civita connection of the Fubini--Study
metric by the flat contact projective structure on $S^{2n+1}$, see
Section 3.4 in \cite{PCS2}. 

In the language of Theorem \ref{thm3.2}, the relevant complexes
correspond to the case that $E_0=S^\ell T^*M$ for some $\ell\in\Bbb N$
and to $k=0$. Hence one starts with a completely symmetric covariant
tensor of valence $\ell$ and $D_0$ is given by taking a covariant
derivative and then completely symmetrizing the result. For the
applications in \cite{E-G} one mainly needs the principal part of the
operator $D_1$ in that complex and the information that, on $\mathbb
CP^n$, one has $\ker(D_1)=\im(D_0)\subset \Ga(E_1)$. Here we indicate
how to get the necessary information on the principal part, the
cohomology of the sequence will be discussed in Section \ref{4} below,
see in particular Theorem \ref{thm4.6}.

We actually discuss a slightly more general setting, looking at the
case that $E_0=S^\ell T^*M$ with an arbitrary even number $k$. In the
language of the proof of Theorem \ref{thm3.2}, the weight determining
$\Bbb V$ is $\la:=k\om_1+\ell\om_2$. The unique element of length two
in the Hasse diagram is given by $\si_1\o\si_2$, where we write
$\si_i$ for the reflection corresponding to the $i$th simple root. The
affine action of this composition maps $\la$ to
$(-k-\ell-3)\om_1+k\om_2+(\ell+1)\om_3$, so this is the weight
corresponding to the bundle $E_2$, which we denote by $E_2^k$ to
indicate the dependence on $k$. The bundle $E_2^0$ is described in
\cite{E-G} in detail. Up to a twist by a natural line bundle, this is
the Cartan product of $\ell+1$ copies of $\La^2T^*M$, i.e.~it
corresponds to the highest weight subspace in
$S^{\ell+1}(\La^2T^*M)$. Hence it can be viewed as tensors with
$2\ell+2$ indices which come up as $\ell+1$ skew symmetric pairs and
the tensor is symmetric under permutations of the pairs of indices.

Representation theory also implies that there is a unique (up to a
constant) natural bundle map $S^{\ell+1}T^*M\otimes S^{\ell+1}T^*M\to
E_2^0$ on conformally symplectic manifolds. This is basically given by
grouping the indices into pairs and then alternating each pair. By
Theorem \ref{thm3.2}, the bundle $E_1$ is, for $k=0$, isomorphic to
$S^{\ell+1}T^*M$. Hence one can use information on BGG sequences on
contact projective structures, to see that $D_1$ has order $\ell+1$
and obtain information on its principal part.

For general $k$, the situation is similar. Up to a twist by a natural
line bundle, $E_2^k$ is the Cartan product of $S^kT^*M$ and $E_2^0$,
while $E_1^k=S^{k+\ell+1}T^*M$. Basic representation theory again
shows that there is a unique (up to scale) natural bundle map
$S^{\ell+1}T^*M\otimes E_1^k\to E_2^k$ on conformally symplectic
manifolds. This shows that $D_1$ still has order $\ell+1$ in the
general case, and one can use results on BGG sequences to get
information on its principal part.

\begin{remark}\label{rem3.2}
Let us briefly discuss the restriction on the parity of the integer
$k$ which determines the order of the first operator in the sequence
in Theorem \ref{thm3.2}. From the proof it is clear that this is only
needed in order that a certain Lie algebra representation integrates
to a group representation of $PSp(2n+2,\Bbb R)$. However, for any
choice of $k$, the Lie algebra representations integrate to
representations of $Sp(2n+2,\Bbb R)$. Hence this restriction could be
avoided if one can construct a parabolic geometry of type
$(Sp(2n+2),P)$ on the contactifications, for example by choosing some
additional data on the given conformally Fedosov structure. 

It seems very plausible that this is possible, at least locally, or
provided that the line bundle $\ell$ defining the conformally
symplectic structure is trivial. However, Section 3.4 of \cite{PCS2}
shows that this does not work in a straightforward way for the global
contactification $S^{2n+1}\to\Bbb CP^n$ defined by the
Hopf--fibration. We will not study this question further here.  
\end{remark}

\subsection{Complexes associated to Bochner--bi--La\-gran\-gean
  metrics}\label{3.3} 

We next discuss the case of PCS--structures associated to simple Lie
algebras of type $A_n$. Here there are two basic structures related to
different real forms of $\frak{sl}(n+2,\Bbb C)$, see Section 3.2 of
\cite{PCS1}. For the split real form $\frak{sl}(n+2,\Bbb R)$, the
corresponding PCS--structure is given by a conformally symplectic
structure $\ell\subset \La^2T^*M$ and a decomposition $TM=E\oplus F$
into a sum of Lagrangean subbundles. Such a structure is torsion--free
if and only if the subbundles $E$ and $F$ in $TM$ are involutive. In
this case, one obtains a para--K\"ahler--metric on $M$ and the
canonical connection for the PCS--structure is the Levi--Civita
connection of this metric.

Parabolic contactification for PCS--structures of this type produces
a so--called Lagrangean contact structure, i.e.~a contact structure
together with a decomposition of the contact subbundle into a direct
sum of Lagrangean subbundles, see Section 4.2.3 of
\cite{book}. Torsion--freeness in this picture again is equivalent to
involutivity of the two Lagrangean subbundles. To obtain differential
complexes from BGG--sequences on the parabolic contactification, we
need this contactification to be locally flat (and thus in particular
torsion--free). By Theorem 7 of \cite{PCS2}, this is the case if
and only if the metric is Bochner--bi--Lagrangean.

To formulate the theorem on BGG sequences in this case, observe that
for $\frak g=\frak{sl}(n+2,\Bbb R)$ the algebra $\frak g_0$ has center
$\Bbb R^2$ and semi--simple part $\frak{sl}(n,\Bbb R)$. Up to twisting
by natural line bundles the subbundles $E,F\subset TM$ correspond to
the standard representation of $\frak{sl}(n,\Bbb R)$ and its
dual. Hence all bundles corresponding to irreducible representations
of $\frak g_0$ can be obtained from natural line bundles and these two
basic bundles via tensorial constructions. Thus we can use a similar
parametrization of BGG sequences as in Theorem \ref{thm3.2}.

\begin{thm}\label{thm3.3}
  Let $\Bbb W$ be an irreducible representation of $\frak{sl}(n,\Bbb
  R)$ and let $W$ be the corresponding natural tensor bundle on a
  PCS--manifold $(M,\ell,E,F)$ of para--K\"ahler type of dimension
  $2n\geq 6$ (with $E$ playing the role of the standard representation
  and $F$ playing the role of its dual). Let $k,\ell\geq 0$ be
  integers such that for even $n$, the number $k+\ell$ is, depending
  on $W$, either even or odd.

Then pushing down an appropriate BGG sequence on parabolic
contactifications leads to a sequence of tensor bundles and invariant
differential operators of the form
$$
\Ga(W_0)\overset{D_0}{\longrightarrow} \Ga(W_1)
\overset{D_1}{\longrightarrow} \cdots
\overset{D_{2n-1}}{\longrightarrow} 
\Ga(W_{2n}) \overset{D_{2n}}{\longrightarrow} 
\Ga(W_{2n+1}).
$$ This sequence is a complex, if $(M,\ell,E,F)$ is
Bochner--bi--Lagrangean. Moreover, the bundles $W_0$ and $W_{2n+1}$
are irreducible, while for $i=1,\dots,n$ the bundles $W_i$ and
$W_{2n+1-i}$ each split into a direct sum of $i+1$ irreducible tensor
bundles. Finally, $W_0$ is the tensor product of $W$ with a natural
line bundle, while $W_1=W_{(1,0)}\oplus W_{(0,1)}$ with
$W_{(1,0)}=S^kE^*\odot W_0$ and $W_{(0,1)}\cong S^\ell F^*\odot W_0$.
\end{thm}
\begin{proof}
Put $\frak g=\frak{sl}(n+2,\Bbb R)$ and let $\frak
g_0^0\cong\frak{sl}(n,\Bbb R)$ be the semisimple part of $\frak
g_0$. Then for the standard numbering $\al_1,\dots,\al_{n+1}$ of
simple roots of $\frak g$, a root space $\frak g_{\al}$ is contained
in $\frak g_0^0$ if and only if $\al$ is a linear combination of
$\al_2,\dots,\al_n$ only. Hence these roots form a simple system for
$\frak g_0^0$. Denoting by $\om_1,\dots,\om_{n+1}$ the fundamental
weights corresponding to the simple system
$\{\al_1,\dots,\al_{n+1}\}$, dominant weights for $\frak g_0^0$ are
equivalent to linear combinations of $\om_2,\dots,\om_n$ with
non--negative integral coefficients. As before, we use negatives of
lowest weights rather than highest weights. Anyway, the irreducible
representation $\Bbb W$ determines a dominant integral weight of
$\frak g_0^0$, which can be written as a linear combination of
$\om_2,\dots,\om_n$ with non--negative integral coefficients.

Adding $k\om_1+\ell\om_{n+1}$ to this weight, we obtain a dominant
integral weight for $\frak g$, which determines an irreducible
representation $\Bbb V$ of $\frak g$. Now we have to discuss whether
the representation $\Bbb V$ integrates to the group $G:=PGL(n+2,\Bbb
R)$, thus giving rise to a tractor bundle and hence to a BGG sequence
on Lagrangean contact structures. For odd $n$, this is not a problem,
since the map $A\mapsto \det(A)^{-1/(n+2)}A$ induces an isomorphism
$PGL(n+2,\Bbb R)\cong SL(n+2,\Bbb R)$ in this case. In the case of
even $n$, $PGL(n,\Bbb R)$ is well known to be isomorphic to
$PSL(n,\Bbb R)$. Hence $\Bbb V$ integrates if and only if the center
$\{\pm\Bbb I\}$ of $SL(n+2,\Bbb R)$ acts trivially on $\Bbb V$. In
terms of the negative of the lowest weight, written as
$a_1\om_1+\dots+a_{n+1}\om_{n+1}$, this boils down to the condition
that the sum of all coefficients with odd indices is even. Depending
on $\Bbb W$, this means that $k+\ell$ either has to be even or has to
be odd.

Having $\Bbb V$ as a representation of $G$, the existence of a BGG
sequence on parabolic geometries of type $(G,P)$ which are equivalent
to Lagrangean contact structures follows from the general theory
developed in \cites{CSS-BGG,Calderbank-Diemer,Rel-BGG2}. Using the
results from Section \ref{2}, this can be pushed down to a sequence of
invariant differential operators on PCS--structures of para--K\"ahler
type. The BGG sequence is a complex if the Lagrangean contact
structure is locally flat, so we obtain a complex on
Bochner--bi--Lagrangean manifolds.

The bundles showing up in the BGG sequence correspond to the Lie
algebra homology groups $H_*(\frak p_+,\Bbb V)$, which can be computed
using Kostant's theorem. In particular, $H_0(\frak p_+,\Bbb V)$ is the
$P$--irreducible quotient of $\Bbb V$ which has the same lowest
weight. This is the tensor product of $\Bbb W$ with the
one--dimensional representation corresponding to
$k\om_1+\ell\om_{n+1}$, which shows that $W_0$ is the tensor product
of $W$ with a natural line bundle.

As noted in the proof of Theorem \ref{thm3.2}, the basic structure of
$H_*(\frak p_+,\Bbb V)$ is encoded in the Hasse diagram $W^{\frak p}$
associated to the parabolic $\frak p$, which is determined in Section
3.6 of \cite{subcomplexes}. In particular, this contains the
information of the number of irreducible components in $H_k(\frak
p_+,\Bbb V)$ for each $k$, thus proving the claims on the number of
irreducible summands in each $W_i$. Finally, the components of
$H_1(\frak p_+,\Bbb V)$ correspond to simple reflections contained in
$W^{\frak p}$. These are exactly the reflections corresponding to
$\al_1$ and $\al_{n+1}$, respectively. Their action on the weight
$k\om_1+a_2\om_2+\dots+a_n\om_n+\ell\om_{n+1}$ is given by adding
$-2(k+1)\om_1+(k+1)\om_2$ and $(\ell+1)\om_n-2(\ell+1)\om_{n+1}$,
respectively. Since these are the negatives of the lowest weights of
$S^{k+1}E^*$ and $S^{\ell+1}F^*$, respectively, the claim on $W_1$
follows.
\end{proof}

\subsection{Remarks on BGG sequences associated to Bochner--K\"ahler
  metrics}\label{3.4}

PACS--structures of K\"ahler type correspond to the real forms
$\frak{su}(p+1,q+1)$ of $\frak{sl}(p+q+2,\Bbb C)$. Such a structure
corresponds to a conformally symplectic structure
$\ell\subset\La^2T^*M$ and an almost complex structure $J$ on $M$, for
which $\ell$ is Hermitian. Torsion--freeness of the structure is
equivalent to $J$ being a complex structure, which then gives rise to
a pseudo--K\"ahler metric of signature $(p,q)$ on $M$. Parabolic
contactifications of PCS--structures of this type are partially
integrable almost CR structures of the appropriate signature, see
Section 4.2.4 of \cite{book}. Torsion--freeness is equivalent to the
structure being integrable and hence a CR structure. By Theorem 7
of \cite{PCS2}, such a parabolic contactification is locally flat if
and only if it is torsion free and the corresponding metric is
Bochner--K\"ahler (of any signature).

As the description suggests, there are strong similarities to
para--K\"ahler type as discussed in Section \ref{3.3} above. In view
of this similarities, we will only briefly outline the differences to
the para--K\"ahler case. 

For $\frak g=\frak{su}(p+1,q+1)$, the subalgebra $\frak g_0$ has
center $\Bbb C$ and semi--simple part $\frak{su}(p,q)$. Up to twisting
by natural line bundles, the tangent bundle $TM$ corresponds to the
standard representation $\Bbb C^{p+q}$ of $\frak g_0$. The analogy to
the para--K\"ahler case becomes clear after complexification, where we
get $TM\otimes\Bbb C=T^{1,0}M\oplus T^{0,1}M$ and the two summands
correspond to dual representations of $\frak{su}(p,q)$. This shows
that, after complexification, the situation is parallel to the
para--K\"ahler case, with complex linearity and anti--linearity
properties replacing the decomposition into $E$ and $F$.

BGG sequences on partially integrable almost CR structures are
associated both to real and to complex representations of the group
$G:=PSU(p+1,q+1)$ which governs the geometry. A complex representation
of $\frak g$ is again determined by its restriction to $\frak g_0$
(which also is a complex representation) and two non--negative
integers describing the action of the center. However, in this case
the center of $SU(p+1,q+1)$ is isomorphic to $\Bbb Z_{p+q+2}$, so the
condition that a representation integrates to $G$ imposes more
restrictive conditions on the two integers describing the action of
the center. Correspondingly, there are less BGG sequences available
then in the para--K\"ahler case, unless it is possible to choose an
additional structure as discussed in Remark \ref{rem3.2}.

Once a complex representation $\Bbb W$ of $\frak g_0$ and two
non--negative integers $k$ and $\ell$ give rise to a complex
representation $\Bbb V$ of $G$, the situation becomes very similar to
Theorem \ref{thm3.3}, compare with Sections 3.6 and 3.8 of
\cite{subcomplexes}. There is a sequence $D_i:\Ga(W_i)\to\Ga(W_{i+1})$
differential operators for $i=0,\dots,2n$, which is a complex provided
that one starts with a Bochner--K\"ahler metric of any signature. The
bundles $W_0$ and $W_{2n+1}$ are irreducible, whereas $W_i$ and
$W_{2n+1-i}$ split into a direct sum of $i+1$ bundles associated to
complex irreducible representations for $i=1,\dots,n$. One may also
describe $W_1$ and the principal parts of the two components of $D_0$
similarly to Theorem \ref{thm3.3}, with complex linearity and conjugate
linearity replacing the appearance of copies of $E$ and $F$.

There are also BGG sequences of partially integrable almost CR
structures induced by real representations of $G$ (which do not admit
a $G$--invariant complex structure). Again, such a representation is
determined by its restriction to $\frak g_0^0$, which is a real
irreducible representation, and by the action of the center. Since
there is no complex structure available, there are stronger
restrictions for the action of the center than in the complex case
here. Next, one needs to make sure that the resulting representation
$\Bbb V$ of $\frak g$ integrates to $G$. Having given a real
representation of $G$, there again is a BGG sequence of the same
length as in the complex case, see again Section 3.8 of
\cite{subcomplexes}. The main difference to the complex case is the
number of irreducible components of the bundles $W_i$ and
$W_{2n+1-i}$, which now is $(i+1)/2$ for odd $i$ and $i/2+1$ for even
$i$. In particular, in this case $W_0$ and $W_1$ both are irreducible
and the principal part of $D_0$ is just given by a symmetrized iterated
covariant derivative followed by a projection to the Cartan product. 

\subsection{Relative BGG sequences}\label{3.5}
We now turn to a second general construction of sequences and
complexes of invariant differential operators on parabolic geometries,
which was introduced in the recent article \cite{Rel-BGG2}. To
simplify comparison to that reference, we briefly change notation and
denote $Q\subset G$ the parabolic subgroup corresponding to a contact
grading. The construction of relative BGG sequences on geometries of
type $(G,Q)$ in addition needs a second parabolic subgroup $P$ such
that $G\supset P\supset Q$. Hence in the case of parabolic contact
structures, this construction is only available for the $A_n$--series,
since for the other series the parabolic subalgebra $\frak
q\subset\frak g$ corresponding to the contact grading is maximal. Even
in the $A_n$--case, an intermediate parabolic is only available for
the real form $\frak{sl}(n+1,\Bbb R)$, for the real forms
$\frak{su}(p+1,q+1)$ the parabolic $\frak q$ corresponding to the
contact grading again is maximal. Hence relative BGG sequences can
only be used to construct invariant differential operators on
PCS--structures of para--K\"ahler type. However, in this case the
resulting sequences are highly interesting since they give rise to
differential complexes under much weaker assumptions than coming from
a Bochner--bi--Lagrangean metric.

Given a type $(G,Q)$ and an intermediate parabolic $P$, the input
needed to construct a relative BGG sequence is a finite dimensional,
completely reducible representation $\Bbb V$ of the group
$P$. Complete reducibility means that the nilpotent subgroup
$P_+\subset P$ acts trivially, so $\Bbb V$ is a representation of the
reductive Levi--factor $P_0\cong P/P_+$. The bundles showing up in the
relative BGG sequence determined by $\Bbb V$ are induced by certain
Lie algebra homology groups which we describe next. The setup easily
implies that $P_+\subset Q_+$ and that $\frak p_+\subset\frak q_+$ is
an ideal. Hence $\frak q_+/\frak p_+$ naturally is a Lie algebra,
which acts on $\Bbb V$ since the restriction to $\frak q_+$ of the
derivative of the $P$--action descends to the quotient. The homology
groups in question then are the groups $H_*(\frak q_+/\frak p_+,\Bbb
V)$, which can be computed algorithmically using a relative version of
Kostant's theorem, see \cite{Rel-BGG1}. As before, we will simply
state the resulting descriptions of bundles in the sequence.

There are general results showing that relative BGG sequences are
complexes under much weaker assumptions than local flatness. The
relevant concept here is called relative curvature. Given a parabolic
geometry $p:\Cal G\to M^\#$ of type $(G,Q)$ and an intermediate
parabolic $P$, the $Q$--invariant subspace $\frak p/\frak
q\subset\frak g/\frak q$ gives rise to a subbundle $T_\rho M^\#\subset
TM$ called the \textit{relative tangent bundle}. Likewise, the
$Q$--invariant subspaces $\frak p_+\subset \frak p\subset\frak g$
corresponds to subbundles $\Cal A_{\frak p_+}M^\#\subset\Cal A_{\frak
  p}M^\#$ in the adjoint tractor bundle $\Cal AM^\#$. One defines the
relative adjoint tractor bundle $\Cal A_{\rho}M^\#$ of the geometry to
be $\Cal G\x_Q(\frak p/\frak p_+)\cong \Cal A_{\frak p}M^\#/\Cal
A_{\frak p_+}M^\#$.

Now the first condition one has to impose is that $T_\rho M^\#\subset
TM^\#$ is an involutive distribution. This is easily seen to be
equivalent to the curvature $\ka\in\Om^2(M,\Cal AM^\#)$ having the
property that its values on two tangent vectors from $T_\rho
M^\#\subset TM^\#$ always lie in $\Cal A_{\frak p}M^\#$. Assuming
this, the values can be projected to the relative adjoint tractor
bundle, thus defining a section $\ka_{\rho}$ of the bundle
$\La^2T_\rho^*M^\#\otimes\Cal A_\rho M^\#$. This is the relative
curvature of the geometry, and if this vanishes identically, any
relative BGG sequence on $M^\#$ is a complex and a fine resolution of
a certain sheaf on $M^\#$, which locally descends to leaf spaces of
the distribution $T_\rho M^\#$.

\subsection{Relative BGG complexes associated to para--K\"ahler
  metrics}\label{3.6} 

In the case of PCS--structures of para--K\"ahler type in dimension
$2n$, the group $G$ is $PGL(n+2,\Bbb R)$ and the parabolic subgroup
$Q\subset G$ comes from the stabilizer of a flag consisting of a line
contained in a hyperplane in the standard representation. Hence
$Q=P\cap \tilde P$, where $P$ comes from the stabilizer of the line
and $\tilde P$ comes from the stabilizer of the hyperplane. Since $P$
and $\tilde P$ are maximal parabolic subgroups in $G$, they are the
only two possible choices of intermediate parabolic subgroups in this
case. A parabolic geometry $p:\Cal G\to M^\#$ of type $(G,Q)$ is given
by a contact structure $H\subset TM^\#$ and a decomposition
$H=E^\#\oplus F^\#$ into a direct sum of Legendrean subbundles. It is
easy to see that the relative tangent bundles corresponding to $P$ and
$\tilde P$ are just the subbundle $E^\#$ and $F^\#$, respectively. In
particular, the situation between $P$ and $\tilde P$ is completely
symmetric, so it suffices to discuss one of the two cases.

On the level of PCS--structures, we have a smooth manifold $M$ of
dimension $2n$, a conformally symplectic structure
$\ell\subset\La^2T^*M$ and a decomposition $TM=E\oplus F$ into
subbundles which are Lagrangean for $\ell$. As discussed in Section
3.2 of \cite{PCS1}, this gives rise to a split--signature conformal
structure on $M$, by extending the pairing between $E$ and $F$ induced
by a local section of $\ell$ to a symmetric tensor field, for which
the two subbundles are isotropic. In the PCS--case, local closed
sections of $\ell$ are uniquely determined up to constant multiples,
so we even get local split signature metrics which are unique up to a
constant factor (and hence all have the same Levi--Civita connection).
In Section 4.5 of \cite{PCS1}, it is shown that torsion freeness of
the PCS--structure defined by $\ell$, $E$ and $F$ is equivalent to the
fact that the subbundles $E$ and $F$ are both involutive. Since the
canonical connection of the PCS--structure preserves the distinguished
metrics by construction, torsion--freeness shows that it has to
coincide with the Levi--Civita connection in this case.

Here we can work in a slightly more general situation, namely that the
one of the subbundles, say $F$, is involutive. Assuming that $E$ is
non--involutive, the canonical connection $\nabla$ of the
PCS--structure has non--trivial torsion (since is preserves $E$). More
precisely, identifying $\La^2T^*M$ with $\La^2E^*\oplus(E^*\otimes
F^*)\oplus\La^2F^*$ the restriction of the torsion to the last two
summands has to be trivial, whereas the restriction to the first
summand coincides with the negative of the tensorial map $\La^2E^*\to
F$ induced by projecting the Lie bracket to $F$. In particular,
$\nabla$ has to be different from the Levi--Civita connection of the
distinguished metrics. However, since $\nabla$ by construction
preserves the distinguished metrics (since it preserves $E$, $F$ and
$\ell$), and its torsion is known, there is an explicit formula
relating it to the Levi--Civita connection.

To formulate the result on complexes induced by relative BGG complexes,
we need a bit more information on the groups involved. Recall that
$G=PGL(n+2,\Bbb R)$, $P\subset G$ comes from the stabilizer of a line
in the standard representation, while $Q\subset P$ comes from the
stabilizer of the flag consisting of that line and a hyperplane
containing it. Via the restriction of the adjoint action of $G$, $P$
acts on on $\frak g/\frak p$ and it is well known that this induces an
isomorphism $P/P_+\cong GL(\frak g/\frak p)\cong GL(n+1,\Bbb R)$,
compare with Section 4.1.5 of \cite{book}. Moreover, the sum of all
but the lowest grading components of $\frak g$ with respect to the
grading defined by $\frak q$ is a codimension--one subspace $\frak
q^{-1}\subset\frak g$ containing $\frak p$. Hence $\frak q^{-1}/\frak
p\subset\frak g/\frak p$ is a hyperplane and $Q\subset P$ can be
characterized as those elements whose action on $\frak g/\frak p$
stabilizes this hyperplane, see Section 4.4.2 of \cite{book}. This
gives rise to a surjection $Q\to GL(\frak q^{-1}/\frak p)\cong
GL(n,\Bbb R)$ which has $Q_+$ in its kernel. For a parabolic contact
structure $(M^\#,H=E^\#\oplus F^\#)$ of type $(G,Q)$, the vector
bundle induced by this representation is $E^\#$. Viewing the above
homomorphism as $Q/Q_+\to GL(\frak q^{-1}/\frak p)$ its kernel is
isomorphic to $\Bbb R\setminus\{0\}$. A faithful representation of
this kernel corresponds to $\La^nF^\#$ on each parabolic contact
structure.

\begin{thm}\label{thm3.6}
  Let $\Bbb W$ be an irreducible representation of $GL(n,\Bbb R)$ and
  let $W$ be the corresponding natural tensor bundle on a
  PCS--manifold $(M,\ell,E,F)$ of para--K\"ahler type of dimension
  $2n$ (with $E$ playing the role of the standard representation) and
  let $k\geq 0$ be an integer.

  Then pushing down an appropriate relative BGG sequence on parabolic
  contactifications leads to a sequence of irreducible tensor bundles
  and invariant differential operators of the form
$$
\Ga(W_0)\overset{D_0}{\longrightarrow} \Ga(W_1)
\overset{D_1}{\longrightarrow} \cdots
\overset{D_{n-2}}{\longrightarrow} \Ga(W_{n-1})
\overset{D_{n-1}}{\longrightarrow} \Ga(W_n).
$$ 
This sequence is a complex, if the subbundle $F\subset TM$ is
involutive. The bundle $W_0$ is the tensor product of $W$ with an real
power of the line bundle $(\La^nF)^2$ and $W_1\cong S^k F^*\odot W_0$.

Finally suppose that $\Bbb W$ is chosen in such a way that $W_0$
coincides with one of the irreducible summands in the bundles from
Theorem \ref{thm3.3}. Then the same holds for all the bundles $W_j$ and
the sequence constructed here is a subsequence respectively a
subcomplex in the sequence from that Theorem.
\end{thm}
\begin{proof}
  Consider $\frak g=\frak{sl}(n+2,\Bbb R)$ with simple roots $\al_i$
  and corresponding fundamental weights $\om_i$ as in the proof of
  Theorem \ref{thm3.3}. The corresponding Cartan subalgebra $\frak
  h\subset\frak g$ also is a Cartan subalgebra for the reductive
  subalgebras $\frak p\cong \frak{gl}(n+1,\Bbb R)$ and $\frak
  q_0\subset\frak p$. Let us decompose $\frak q_0\cong\frak{gl}(n,\Bbb
  R)\oplus\Bbb R$ as described on the group level before the
  theorem. Then the negative of the lowest weight of the
  representation $\Bbb W$ can be expressed as a linear combination
  $a_1\om_1+\dots +a_n\om_n$ with $a_1\in\Bbb R$ and non--negative
  integers $a_2,\dots,a_n$. Adding $k\om_{n+1}$ to this, we obtain a
  weight which is the negative of the lowest weight of a finite
  dimensional, irreducible representation $\Bbb V$ of $\frak p$. (The
  part $a_2\om_2+\dots+a_n\om_n+k\om_{n+1}$ is a dominant integral
  weight for the semisimple part of $\frak p$, and adding $a_1\om_1$
  corresponds to tensorizing by a one--dimensional representation of
  the center.) There is no problem with the representation integrating
  to the group $P\cong GL(n+1,\Bbb R)$.

Hence the general results of \cite{Rel-BGG2} imply the existence of an
associated relative BGG--sequence for each parabolic geometry of type
$(G,Q)$. Via the mechanism introduced in Section \ref{2}, this
sequence can be pushed down a manifold endowed with a PCS--structure
from local contactifications. The bundles showing up in the resulting
sequence are induced by the Lie algebra homology groups $H_k(\frak
q_+/\frak p_+,\Bbb V)$, in particular the degrees range from $0$ to
$\dim(\frak q_+)-\dim(\frak p_+)=n$. To obtain the shape of the
sequence, one has to determine the relative Hasse diagram $W^{\frak
  q}_{\frak p}$ as described in Lemma 2.6 and Example 3.2 of
\cite{Rel-BGG1}. It is easy to see $W^{\frak q}_{\frak p}$ consists of
$n+1$ elements of length $0,\dots,n$. This implies the statement on
irreducibility of $W_i$ for each $i$. Moreover, $H_0(\frak
q_+/\frak p_+,\Bbb V)$ is the $Q$--irreducible quotient of $\Bbb V$,
which implies the description of $W_0$. The unique element of length
$1$ in $W^{\frak p}_{\frak q}$ is the simple reflection corresponding
to $\al_{n+1}$, from which the description of $W_1$ follows as in the
proof of Theorem \ref{thm3.3}. 

Suppose next, that the subbundle $F\subset TM$ is involutive. Then for
each local parabolic contactification $(M^\#,H=E^\#\oplus F^\#)$, the
subbundle $F^\#\subset TM^\#$ is involutive, too. As we have observed
above, this is exactly the relative tangent bundle $T_\rho M^\#$ for
the intermediate parabolic $P$. By Proposition 4.2.3 of \cite{book}
this implies that one of the three harmonic curvature components of
the parabolic geometry on $M^\#$ vanishes identically. But the
discussion of harmonic curvature in Section 4.2.3 of \cite{book} shows
that the assumptions of part (1) of Proposition 4.18 of
\cite{Rel-BGG2} are satisfied, so the relative curvature of the
geometry vanishes. By part (1) of Theorem 4.11 of that reference, any
relative BGG sequence on $M^\#$ is a complex, so the descended
sequence is a complex, too. 

To prove that last claim, we observe that by Kostant's theorem, all
$\frak p$--dominant weights in the affine Weyl orbit of the negative
of the lowest weight of $\Bbb V$ are realized by irreducible
components of the representations $H_j(\frak p_+,\Bbb V)$ for
$j=0,\dots,\dim(\frak p_+)$. Hence our assumptions mean that $W_0$
occurs as an irreducible component in one of these homology
representations (which happen to be irreducible in our case). But by
Theorem \ref{thm3.3} of \cite{Rel-BGG1}, the homologies $H_i(\frak
q_+/\frak p_+,H_j(\frak p_+,\Bbb V))$ are contained in $H_{i+j}(\frak
q_+,\Bbb V)$, so all bundles $W_i$ occur in the sequence from Theorem
\ref{thm3.3}. It is proved in Theorem 5.2 of \cite{Rel-BGG2} that then
the absolute and the relative BGG constructions produce the same
differential operators between these bundles, which implies the last
claim. 
\end{proof}

\begin{remark}\label{rem3.6}
  As already remarked above, relative BGG sequences are not available
  for the parabolic contact geometries associated to $\frak
  g=\frak{su}(p+1,q+1)$. However, there is a case in which our methods
  can produce differential complexes on general K\"ahler manifolds
  (i.e.~without the assumption on vanishing Bochner curvature). This
  is related to those cases in Theorem \ref{thm3.6} in which a relative
  BGG sequence is included in a proper BGG sequence. In these cases,
  the existence of subcomplexes can also be proved more directly, see
  \cite{subcomplexes}. While these techniques require slightly
  stronger assumption (involutivity of both $E$ and $F$), they also
  work for the other real forms. 

  In the setting of Section \ref{3.4} this method applies to
  torsion--free PCS--structures of K\"ahler type, which are equivalent
  to pseudo--K\"ahler metrics of any signature, see Proposition 4.5 of
  \cite{PCS1}. The local contactifications of such a geometry carry an
  (integrable) CR structure of hypersurface type of the same
  signature. As discussed in Section \ref{3.4} there are BGG sequences
  on such structures associated to real and complex representations of
  the groups $PSU(p+1,q+1)$. Theorem 3.8 in \cite{subcomplexes} shows
  that, both in the real and in the complex case, there are several
  subcomplexes in such a BGG sequences, which descend to differential
  complexes on the underlying pseudo--K\"ahler manifolds.
\end{remark}

\section{The cohomology of descended BGG complexes}\label{4} 
In this last section, we will derive some results on the cohomology of
the differential complexes associated to special symplectic
connections via descending BGG sequences. The strongest results are
obtained in the case of global contactifications with compact fibers,
thus in particular applying to complexes on $\Bbb CP^n$ and
$Gr(2,\Bbb C^n)$ as treated in Sections 2.6 and 3.4 of
\cite{PCS2}. Several steps towards these main results are proved in a
more general setting.

\subsection{The relation to twisted de--Rham cohomology}\label{4.1} 
For the first step in the description, we need a few details on the
construction of BGG sequences. As discussed in Section \ref{3.1}, a
BGG sequence on parabolic contact structures of type $(G,P)$ is
determined by a representation $\Bbb V$ of the Lie group $G$. Via
taking the associated bundle to the Cartan bundle determined by $\Bbb
V$, this representation gives rise to a natural vector bundle on such
geometries. On the homogeneous model $G/P$, this is just the
homogeneous vector bundle $G\x_P\Bbb V$. Bundles of this type are
called \textit{tractor bundles}. Their main feature is that the
Cartan connection induces a canonical linear connection on each
tractor bundle, which is flat if and only if either $\Bbb V$ is a
trivial representation or the geometry is locally flat. Since the case
of the trivial representation is treated in \cite{Cap-Salac}, we will
always assume that $\Bbb V$ is a non--trivial, irreducible
representation from now on.

Given a parabolic contact geometry $(p:\Cal G^\#\to M^\#,\om)$, let us
denote by $\Cal VM^\#$ the tractor bundle on $M^\#$ induced by $\Bbb
V$ and by $\nabla^{\Cal V}$ the canonical tractor connection induced
by $\om$. Coupling $\nabla^{\Cal V}$ to the exterior derivative, one
obtains the \textit{covariant exterior derivative}
$d^\nabla:\Om^k(M^\#,\Cal VM^\#)\to\Om^{k+1}(M^\#,\Cal VM^\#)$ for
each $k$. Obviously, this is a sequence of invariant differential
operators and it is well known that they form a complex if and only if
the connection $\nabla^{\Cal V}$ is flat. We will refer to this as the
twisted de--Rham sequence respectively the twisted de--Rham complex
determined by $\Bbb V$.

Now it turns out that, via the Cartan connection $\om$, the cotangent
bundle $T^*M^\#$ can be naturally identified with the associated
bundle $\Cal G^\#\x_P\frak p_+$. Hence the bundle of $\Cal
VM^\#$--valued $k$--forms is induced by the representation $\La^k\frak
p_+\otimes\Bbb V$, which is the space of $k$--chains in the standard
complex computing the Lie algebra homology $H_*(\frak p_+,\Bbb
V)$. Since the standard differentials in this complex are
$P$--equivariant, they induce natural bundle maps
$\La^kT^*M^\#\otimes\Cal VM^\#\to \La^{k-1}T^*M^\#\otimes\Cal VM^\#$,
which traditionally are denoted by $\partial^*$. Hence
$\Im(\partial^*)$ and $\ker(\partial^*)$ are nested natural subbundles
in $\La^kT^*M\otimes\Cal VM$ and the quotient
$\ker(\partial^*)/\Im(\partial^*)$ is by construction isomorphic to
the associated bundle $\Cal G\x_PH_k(\frak p_+,\Bbb V)$ which we
denote by $\Cal H^{\Cal V}_kM^\#$.

By construction, there is a bundle projection
$\Pi=\Pi_k:\ker(\partial^*)\to \Cal H^{\Cal V}_kM^\#$ which induces a
tensorial operator on the spaces of sections of these bundles that we
denote by the same symbol. Now the key to the construction of BGG
sequences is that there is an invariant differential operator $S=S_k$
which splits this tensorial projection. Otherwise put, to any section
$\si\in\Ga(\Cal H^{\Cal V}_kM^\#)$ we can associate
$S(\si)\in\Om^k(M^\#,\Cal VM^\#)$ such that $\partial^*\o S(\si)=0$
and $\Pi(S(\si))=\si$. Moreover it turns out that this splitting
operator has the property that $\partial^*\o d^\nabla(S(\si))=0$ for
any $\si$ (which uniquely determines $S$). Hence one can define an
invariant differential operator $D^\#=D_k^\#:\Ga(\Cal H^{\Cal
  V}_kM^\#)\to\Ga(\Cal H^{\Cal V}_{k+1}M^\#)$ by
$D^\#(\si):=\Pi(d^\nabla(S(\si)))$, and these operators form the BGG
sequence. Moreover, if $\nabla^{\Cal V}$ is flat, then the splitting
operators have the property that $d^\nabla\o S_k=S_{k+1}\o D^\#_k$ for
all $k$. This readily implies that the BGG sequence is a complex, and
the $S_k$ define a homomorphism of complexes from the BGG complex to
the twisted de--Rham complex.

Now suppose that we have given a PCS--quotient $q:M^\#\to M$ of type
$(G,P)$, and let $\pi:\Cal G_0\to M$ be the corresponding
$G_0$--principal bundle. Then by construction, the operators $D$
obtained by descending the operators $D^\#$ act on sections of the
bundle $\Cal H^{\Cal V}_kM:=\Cal G_0\x_{G_0}H_k(\frak p_+,\Bbb V)$ for
all $k$. (Here one uses that each $H_k(\frak p_+,\Bbb V)$ is a
completely reducible representation of $P$, so it descends to
$P/P_+\cong G_0$.) Recall that the infinitesimal automorphism
$\xi\in\frak X(M^\#)$ giving rise to the PCS--quotient induces a
$P$--invariant vector field $\tilde\xi\in\frak X(\Cal G^\#)$ which
projects onto $\xi$. Similarly as discussed in Section \ref{2.4},
sections of associated bundles to $\Cal G^\#$ can be identified with
smooth functions with values in the inducing representation, so there
is a natural action of $\tilde\xi$ via a Lie derivative $\Cal
L_{\tilde\xi}$. In particular, we define $\Om^k_\xi (M^\#,\Cal VM^\#)$
as the subspace of those forms $\ph$, for which $\Cal
L_{\tilde\xi}(\ph)=0$. This works in the same way on open subsets of
$M^\#$ so that we have actually defined a subsheaf of the sheaf of
$\Cal VM^\#$--valued $k$--forms.

\begin{thm}\label{thm4.1}
Consider a PCS--quotient $q:M^\#\to M$ of type $(G,P)$. Let $\Bbb V$
be a representation of $G$, $\Cal VM^\#\to M^\#$ the tractor bundle
determined by $\Bbb V$ and $(\Om^*(M^\#,\Cal VM^*),d^\nabla)$ the
induced twisted de--Rham sequence. Let $(\Cal H^{\Cal V}_*M,D_*)$ be
the sequence of differential operators on $M$ obtained by descending
the BGG sequence determined by $\Bbb V$ as in Theorem \ref{thm2.4}.

Then $d^\nabla$ commutes with $\Cal L_{\tilde\xi}$, and hence it
preserves the subspaces $\Om^*_\xi(M^\#,\Cal VM^\#)$. Moreover, if
$\nabla$ is flat, then these subspaces form a subcomplex in the twisted
de--Rham complex, whose cohomology is naturally isomorphic to the
cohomology of the complex $(\Cal H^{\Cal V}_*M,D_*)$.
\end{thm}
\begin{proof}
Since any local flow of $\tilde\xi\in\frak X(\Cal G^\#)$ is an
automorphism of the Cartan geometry $(p:\Cal G^\#\to M^\#,\om)$, it
follows as in Lemma \ref{lem2.4} that any invariant differential operator
commutes with $\Cal L_{\tilde\xi}$. In particular, applying this to
$d^\nabla$ we readily conclude that $d^\nabla(\Om^k_\xi(M^\#,\Cal
VM^\#))\subset \Om^{k+1}_\xi(M^\#,\Cal VM^\#)$. In the case that
$\nabla$ is flat, we hence get a subcomplex in the twisted de--Rham
complex. 

We can also apply this argument to the BGG operators $D^\#$. For a
natural vector bundle $\Cal WM^\#$ let us denote by $\Ga_{\xi}(\Cal
WM^\#)\subset \Ga(\Cal WM)$ the kernel of $\Cal
L_{\tilde\xi}$. Naturality of the BGG operators then shows that we get
a subcomplex $(\Ga_\xi(\Cal H_*^{\Cal V}M^\#),D^\#_*)$ in the BGG
complex. Now since each $H_k(\frak p_+,\Bbb V)$ is a completely
reducible representation of $P$, we can equivalently describe sections
of the corresponding associated bundle via the intermediate principal
bundle $\Cal G_0^\#:=\Cal G^\#/P_+$. Recall from Section \ref{2.2}
that there is a vector field $\xi_0\in\frak X(\Cal G_0^\#)$ which lies
between $\tilde\xi$ and $\xi$. By construction, identifying $\Ga(\Cal
H_k^{\Cal V}M^\#)$ with $C^\infty (\Cal G_0,H_k(\frak p_+,\Bbb
V))^{G_0}$, the subspace $\Ga_\xi$ exactly corresponds to the kernel
of the Lie derivative $\Cal L_{\xi_0}$. Theorem \ref{thm2.4} and Lemma
\ref{lem2.4} thus imply that the complex $(\Ga_\xi(\Cal H_*^{\Cal
  V}M^\#),D^\#_*)$ is isomorphic to $(\Ga(\Cal H_*^{\Cal V}M),D_*)$.

Thus we can complete the proof by showing that $(\Ga_\xi(\Cal
H_*^{\Cal V}M^\#),D^\#_*)$ computes the same cohomology as the
subcomplex $(\Om^*_\xi(M^\#,\Cal VM^\#),d^\nabla)$ in the twisted
de--Rham complex. For this, we can adapt the usual proof for BGG
sequences from Theorem 2.6 and Lemma 2.7 of \cite{CSS-BGG}. Naturality
of the splitting operators implies that for each $k$ we get
$S(\Ga_\xi(\Cal H_k^{\Cal V}M^\#))\subset
\Ga_\xi(\ker(\partial^*))\subset \Om_\xi^k(M^\#,\Cal VM^\#)$.  In
particular, the fact that $d^\nabla\o S=S\o D^\#$ verified in Lemma
2.7 of \cite{CSS-BGG} shows that $S$ defines a complex map between the
two subcomplexes, and we claim that this induces an isomorphism in
cohomology. Suppose that $\ph\in\Om^k_\xi(M^\#,\Cal VM^\#)$ satisfies
$d^\nabla\ph=0$. Then in Lemma 2.7 of \cite{CSS-BGG} it is shown that
there is a form $\ps\in\Om^{k-1}(M^\#,\Cal VM^\#)$ such that
$\ph+d^\nabla\ps\in\Ga(\ker(\partial^*))$. As shown in Theorems 3.9
and 3.14 of \cite{Rel-BGG2}, a form with this property can be obtained
as the value of an invariant differential operator on $\ph$. Thus we may
assume that $\Cal L_{\tilde\xi}\ps=0$ and hence
$\ph+d^\nabla\ps\in\Ga_\xi(\ker(\partial^*))$. But then naturality of
the bundle map $\Pi$ shows that
$\al:=\Pi(\ph+d^\nabla\ps)\in\Ga_\xi(\Cal H_k^{\Cal V}M^\#)$. Now by
construction $d^\nabla(\ph+d^\nabla\ps)=0$ which shows that
$\ph+d^\nabla\ps=S(\al)$ and $D^\#(\al)=0$. Hence the cohomology class
of $\al$ is mapped to the cohomology class of $\ph$, so the induced
map in cohomology is surjective.

On the other hand, suppose that $\al\in\Ga_\xi(\Cal H_k^{\Cal V}M^\#)$
satisfies $D^\#(\al)=0$ and $S(\al)=d^\nabla\ps$ for some
$\ps\in\Om^{k-1}_\xi(M^\#,\Cal VM^\#)$. As in the previous step, we
may without loss of generality assume that
$\ps\in\Ga_\xi(\ker(\partial^*))$ and then project this to
$\be=\Pi(\ps)\in\Ga_\xi(\Cal H_{k-1}^{\Cal V}M^\#)$. Then
$d^\nabla\ps=S(\al)\in\Ga(\ker(\partial^*))$ shows that $\ps=S(\be)$
and hence $D^\#(\be)=\Pi\o S(\al)=\al$, which shows injectivity of the
induced map in cohomology.
\end{proof}

\subsection{Reduction to horizontal equivariant forms}\label{4.2}  
Suppose that $q:M^\#\to M$ is a PCS--quotient with corresponding
infinitesimal automorphism $\xi\in\frak X(M^\#)$ corresponding to
$\tilde\xi\in\frak X(\Cal G^\#)$. Then the forms in $\Om^*_\xi(M^\#,\Cal
VM^\#)$ as studied above, are (in an appropriate sense) equivariant
for $\tilde\xi$. Similarly to the case of ordinary forms treated in
Section 2.3 of \cite{Cap-Salac}, it is natural to next look at forms
which in addition are horizontal, since these essentially are objects
on $M$ already. Hence we define $\Om^k_\xi(M^\#,\Cal
VM^\#)_{\text{hor}}$ to be the space of those $\ph\in\Om^k(M^\#,\Cal
VM^\#)$, for which $\Cal L_{\tilde\xi}\ph=0$ and $i_\xi\ph=0$. To
simplify notation, we will write $A^k:=\Om^k_\xi(M^\#,\Cal VM^\#)$ and
$A^k_\hor:=\Om^k_\xi(M^\#,\Cal VM^\#)_{\text{hor}}$ in what follows.

Recall that the infinitesimal automorphism $\xi$ determines a unique
contact form $\al\in\Om^1(M^\#)$ for which $\xi$ is the Reeb field,
i.e.~such that $i_\xi\al=1$ and $i_\xi d\al=0$, see Proposition 2.2 of
\cite{Cap-Salac}. Observe also, that there is an obvious wedge product
$\Om^k(M^\#)\x\Om^\ell(M^\#,\Cal VM^\#)\to\Om^{k+\ell}(M^\#,\Cal
VM^\#)$. In terms of these operations, we can now derive a description
of $A^k$.

\begin{lemma}\label{lem4.2}
  Let $q:M^\#\to M$ be a PCS--quotient of type $(G,P)$ with
  corresponding infinitesimal automorphism $\xi\in\frak X(M^\#)$ and let
  $\al\in\Om^1(M^\#)$ be the contact form associated to $\xi$. Then
  the maps $\ph\mapsto (\ph-\al\wedge i_\xi\ph,i_\xi\ph)$ and
  $(\ph_1,\ph_2)\mapsto (\ph_1+\al\wedge\ph_2)$ define inverse
  isomorphisms between $A^k$ and $A^k_\hor\oplus A^{k-1}_\hor$.
\end{lemma}
\begin{proof}
  Since $i_\xi\o i_\xi=0$, we see that both $i_\xi\ph$ and
  $\ph-\al\wedge i_\xi\ph$ are horizontal, and then one immediately
  verifies that the two maps in the claim are inverse to each
  other. So it remains to show that the construction can be restricted
  to the kernels of $\Cal L_{\tilde\xi}$ on both sides.

  To do this, we have to derive some results on the operator $\Cal
  L_{\tilde\xi}$, which by definition is given by differentiating the
  equivariant functions corresponding to sections of natural vector
  bundles in the direction of $\tilde\xi$. Let us denote by $(p:\Cal
  G^\#\to M^\#,\om)$ the Cartan geometry describing the parabolic
  contact structure on $M^\#$. Then the isomorphism $TM^\#\cong\Cal
  G^\#\x_P(\frak g/\frak p)$ comes from the fact that for $u\in\Cal G$
  the map $\om(u):T_u\Cal G^\#\to\frak g$ descends to a linear
  isomorphism $T_{p(u)}M^\#\to\frak g/\frak p$. Otherwise put, the
  equivariant smooth function $f:\Cal G^\#\to\frak g/\frak p$
  corresponding to a vector field $\eta\in\frak X(M^\#)$ can be
  written as $\om(\tilde\eta)+\frak p$, where $\tilde\eta\in\frak
  X(\Cal G^\#)$ is a $P$--equivariant lift of $\eta$. Since $\xi$ is
  an infinitesimal automorphism, the lift $\tilde\xi\in\frak X(\Cal
  G^\#)$ satisfies $0=\Cal L_{\tilde\xi}\om$. This implies that for
  $\tilde\eta$ as above, we get
  $\tilde\xi\cdot\om(\tilde\eta)=\om([\tilde\xi,\tilde\eta])$. Since
  $[\tilde\xi,\tilde\eta]$ is a lift of $[\xi,\eta]$, we conclude that
  $\Cal L_{\tilde\xi}\eta=[\xi,\eta]$. Hence on $\frak X(M)$ the
  operator $\Cal L_{\tilde\xi}$ coincides with the usual Lie
  derivative $\Cal L_\xi$ along $\xi$, and in particular, $\Cal
  L_{\tilde\xi}\xi=0$.

  By construction, $\Cal L_{\tilde\xi}$ satisfies the usual
  compatibility conditions with tensor products and
  contractions. Using this, the result for vector fields easily
  implies that $\Cal L_{\tilde\xi}$ coincides with the usual Lie
  derivative $\Cal L_\xi$ on all tensor fields and in particular on
  (real valued) differential forms. The definition of the contact form
  $\al$ then implies that $0=\Cal L_\xi\al=\Cal
  L_{\tilde\xi}\al$. Together with naturality and $\Cal
  L_{\tilde\xi}\xi=0$, this now implies that all the maps we have used
  preserve the kernels of $\Cal L_{\tilde\xi}$.
\end{proof}

For the next step, we have to impose an additional restriction on the
infinitesimal automorphism in question. Since $\tilde\xi\in\frak
X(\Cal G^\#)$ is $P$--invariant vector field, equivariancy of the
Cartan connection $\om$ implies that $\om(\tilde\xi):\Cal G^\#\to\frak
g$ is a $P$--equivariant function. Thus it defines a smooth section of
the associated bundle $\Cal AM^\#:=\Cal G^\#\x_P\frak g$, the adjoint
tractor bundle of the parabolic geometry $(p:\Cal G^\#\to
M^\#,\om)$. Indeed, this establishes a bijection between $\Ga(\Cal
AM^\#)$ and the space of $P$--invariant vector fields on $\Cal
G^\#$. It turns out that infinitesimal automorphisms can be nicely
characterized in this picture, see \cite{deformations}. Since $\Cal
AM^\#$ is a tractor bundle, it carries the tractor connection
$\nabla^{\Cal A}$. It turns out (see Proposition 3.2 of
\cite{deformations}) that this connection can be naturally modified by
a term involving the Cartan curvature to a linear connection
$\tilde\nabla$ whose parallel sections exactly correspond to the
canonical lifts $\tilde\xi\in\frak X(\Cal G^\#)$ of infinitesimal
automorphisms $\xi\in\frak X(M^\#)$. This bijection is implemented by
an analog of the splitting operator $S$ discusses in Section \ref{4.1}
above. Moreover, it turns out that if a section of $\Cal AM^\#$ is
parallel for $\nabla^{\Cal A}$, then it is also parallel for
$\tilde\nabla$, see Corollary 3.5 of \cite{deformations}. Hence
parallel sections of $\nabla^{\Cal A}$ correspond to a subclass of
infinitesimal automorphisms.
\begin{definition}\label{def4.2}
  Let $(p:\Cal G^\#\to M^\#,\om)$ be a parabolic geometry of type
  $(G,P)$. An infinitesimal automorphism $\xi\in\frak X(M^\#)$ of the
  geometry is called \textit{normal} if and only if the induced
  $P$--invariant vector field $\tilde\xi\in\frak X(\Cal G^\#)$
  corresponds to a section of $\Cal AM^\#$ which is parallel for the
  tractor connection $\nabla^{\Cal A}$. 
\end{definition}
By Corollary 3.5 of \cite{deformations} an infinitesimal automorphism
$\xi$ is normal if and only if $\xi$ inserts trivially into the
curvature two--form of the Cartan connection $\om$. In particular, any
infinitesimal automorphism on a locally flat geometry is normal. 

\medskip

Next, we can use the infinitesimal automorphism $\xi$ and its lift
$\tilde\xi$ to define a smooth bundle map $\Xi:\Cal VM^\#\to\Cal
VM^\#$ on a tractor bundle $\Cal VM^\#\to M^\#$. To define this,
observe that $\Cal VM^\#=\Cal G^\#\x_P\Bbb V$ for a representation
$\Bbb V$ of $G$, so we have the infinitesimal representation $\frak
g\to L(\Bbb V,\Bbb V)$. This means that any point $u\in\Cal G^\#$
defines a linear isomorphism $\ps_u:\Bbb V\to \Cal V_xM^\#$, where
$x=p(u)\in M^\#$. For any $g\in P$ and $v\in\Bbb V$, we then get
$\ps_{u\cdot g}(v)=\ps_u(g\cdot v)$, so $\ps_{u\cdot
  g}=\ps_u\o\rho(g)$, where $\rho$ denotes the representation of
$G$. On the other hand, the function $\om(\tilde\xi):\Cal G^\#\to\frak
g$ satisfies $\om(\tilde\xi)(u\cdot
g)=\Ad(g^{-1})(\om(\tilde\xi)(u))$. This shows that, denoting by
$\rho'$ the infinitesimal representation, we conclude that
$$
\ps_u\o \rho'(\om(\tilde\xi)(u))\o\ps_u^{-1}=\ps_{u\cdot g}\o
\rho'(\om(\tilde\xi)(u\cdot g))\o\ps_{u\cdot g}^{-1}.
$$
Thus we get a well defined linear map $\Xi(x):\Cal V_xM^\#\to\Cal
V_xM^\#$ and hence a smooth bundle map as claimed. 

\begin{prop}\label{prop4.2}
  Let $(p:\Cal G^\#\to M^\#,\om)$ be a parabolic geometry of type
  $(G,P)$ such that $M^\#$ is connected. Let $\xi\in\frak X(M^\#)$ be
  a normal infinitesimal automorphism and let $\Xi:\Cal VM^\#\to\Cal
  VM^\#$ be the induced bundle map on a tractor bundle $\Cal VM^\#\to
  M^\#$. Then $\Xi$ has constant rank, so its kernel $\ker(\Xi)$ and
  its image $\Im(\Xi)$ are smooth subbundles of $\Cal VM^\#$, and we
  get a smooth vector bundle $\Coker(\Xi):=\Cal VM^\#/\Im(\Xi)$.
\end{prop}
\begin{proof}
  In Lemma 2.3 of \cite{hol-red} it is shown that connectedness of
  $M^\#$ implies that for a normal infinitesimal automorphism $\xi$, the
  image of the function $\om(\tilde\xi):\Cal G^\#\to\frak g$ is
  contained in a single orbit of the adjoint action of $g$. Now for
  $g\in G$ and $X\in\frak g$, and the actions $\rho$ of $G$ and
  $\rho'$ of $\frak g$ on $\Bbb V$, it is well known that
  $\rho'(\Ad(g)(X))=\rho(g)\o\rho'(X)\o\rho(g)^{-1}$. This shows that
  the maps $\rho'(X)$ and $\rho'(\Ad(g)(X))$ have the same rank, which
  by construction implies that $\Xi$ has constant rank. All other
  claims are well known consequences of this fact.  
\end{proof}

Using $\Xi$, we can now define several subspaces in the space of $\Cal
VM$--valued differential forms. First, we of course have
$\Om^k(M^\#,\ker(\Xi))\subset\Om^k(M^\#,\Cal VM^\#)$. Moreover, even
though $\ker(\xi)$ is not a natural vector bundle, it makes no problem
to require $\Cal L_{\tilde\xi}\ph=0$ as well as $i_\xi\ph=0$ for
$\ph\in\Om^k(M^\#,\ker(\Xi))$. Thus we can define
$K^i:=\Om^i_\xi(M^\#,\ker(\Xi))_{hor}\subset A^i_{hor}$. The cokernel
of $\Xi$ is more complicated to deal with, and we need some
preliminary results to do this.

\subsection{The Cartan formula}\label{4.3}
We next derive an analog of the Cartan formula for the covariant
exterior derivative $d^\nabla$ on $\Om^*(M^\#,\Cal VM^\#)$. This will
be a crucial steps towards the construction of various subcomplexes
and to a description of the cohomology of the subcomplex from Theorem
\ref{thm4.1}. 

\begin{lemma}\label{lem4.3}
For any $\ph\in\Om^*(M^\#,\Cal VM^\#)$ we have 
$$
\Cal L_{\tilde\xi}\ph=i_\xi d^\nabla\ph+d^\nabla
i_\xi\ph-\Xi_*(\ph), 
$$
where $\Xi_*:\Om^*(M^\#,\Cal VM^\#)\to\Om^*(M^\#,\Cal VM^\#)$ is given
by applying $\Xi$ to the values of $\Cal VM^\#$--valued forms.  

In particular, for $\ph\in A^*_\hor$, we have $i_\xi
d^\nabla\ph=\Xi_*(\ph)$. 
\end{lemma}
\begin{proof}
The first step is as in the proof of Cartan's formula for the exterior
derivative. Using the standard formula for $d^\nabla$, one verifies
that the value of $i_\xi d^\nabla\ph+d^\nabla i_\xi\ph$ maps vector
fields $\eta_1,\dots,\eta_k\in\frak X(M^\#)$ to
\begin{equation}\label{iddi}
\nabla^{\Cal VM}_\xi\ph(\eta_1,\dots,\eta_k)+\textstyle\sum_{i=1}^k(-1)^{i}
\ph([\xi,\eta_i],\eta_1,\dots,\widehat{\eta_i},\dots,\eta_k).
\end{equation}
Let us denote by $s\in\Ga(\Cal AM^\#)$ the section of the adjoint
tractor bundle corresponding to $\tilde\xi\in\frak X(\Cal
G^\#)^P$. Then by the construction from Section \ref{4.1}, the
operator $\Cal L_{\tilde\xi}$ coincides with the so--called
fundamental derivative $D_s$, see Section 1.5.8 of
\cite{book}. Likewise, the bundle map $\Xi$ by construction coincides
with the operation $s\bullet\ $ from Section 1.5.7 of
\cite{book}. Thus the formula for the tractor connection in Theorem
1.5.8 of \cite{book} shows that the first summand in \eqref{iddi} can
be rewritten as
\begin{equation}\label{nabla}
\Cal L_{\tilde\xi}(\ph(\eta_1,\dots,\eta_k))+
\Xi(\ph(\eta_1,\dots,\eta_k)). 
\end{equation}
In the second part of \eqref{iddi}, we can move the Lie bracket to the
$i$th entry of $\ph$ at the expense of a sign $(-1)^{i-1}$. As noted
in the proof of Lemma \ref{lem4.2}, we have $[\xi,\eta_i]=\Cal
L_{\tilde\xi}\eta_i$, so the second term in \eqref{iddi} can be
written as 
$$
-\textstyle\sum_{i=1}^k\ph(\eta_1,\dots,\Cal
L_{\tilde\xi}\eta_i,\dots,\eta_k), 
$$
and naturality of $\Cal L_{\tilde\xi}$ implies that this adds up with
the first term in \eqref{nabla} to $(\Cal
L_{\tilde\xi}\ph)(\eta_1,\dots,\eta_k)$. 
\end{proof}

The last statement of the Lemma shows that $d^\nabla$ does not
preserve the subspace $\Om^*_\xi(M^\#,\Cal VM^\#)_{hor}$. Of course
there is the possibility of combining $d^\nabla$ with the projection
to horizontal forms from Lemma \ref{lem4.2}:

\begin{definition}\label{def4.3}
We define the \textit{horizontal derivative} 
$$
\hat d:\Om^k(M^\#,\Cal VM^\#)\to \Om^{k+1}(M^\#,\Cal VM^\#)
$$
by $\hat d\ph=d^\nabla\ph-\al\wedge i_\xi d^\nabla\ph$. 
\end{definition}

Notice that by definition $i_\xi\hat d\ph=0$ for all
$\ph\in\Om^*(M^\#,\Cal VM^\#)$. Moreover, all the operations used in
the definition are compatible with $\Cal L_{\tilde\xi}$, so $\hat d
(A^k)\subset A^{k+1}_\hor$.

\begin{prop}\label{prop4.3}
  Assuming that $\xi\in\frak X(M^\#)$ is a normal infinitesimal
  automorphism, we have:

  (1) The operator $\Xi_*$ commutes with $d^\nabla$ and with $\hat d$.

  (2) If $d^{\nabla}\o d^\nabla=0$, then the subspaces
  $K^k=\Om^k_\xi(M^\#,\ker(\Xi))_{hor}$ form a subcomplex of
  $(A^*,d^\nabla)$.

  (3) Denoting by $C^k$ the quotient $A^k_\hor/\Xi_*(A^k_\hor)$, the
  horizontal derivative induces a well defined operator $\hat d:C^k\to
  C^{k+1}$ for each $k$. If $d^{\nabla}\o d^\nabla=0$, then $(C^*,\hat
  d)$ is a complex.
\end{prop}
\begin{proof}
  (1) Let $s\in\Ga(\Cal AM^\#)$ be the section corresponding to
  $\tilde\xi\in\frak X(\Cal G^\#)$, so by assumption $\nabla^{\Cal
    A}s=0$. As we have noted in the proof of Lemma \ref{lem4.3} above, we
  get $\Xi_*\ph(\eta_1,\dots,\eta_k)=s\bullet
  (\ph(\eta_1,\dots,\eta_k))$ for arbitrary vector fields
  $\eta_1,\dots,\eta_k\in\frak X(M^\#)$. Using Proposition 1.5.7 of
  \cite{book}, this shows that
  $$
  \nabla_\eta^{\Cal
    V}(\Xi_*\ph(\eta_1,\dots,\eta_k))=\Xi_*(\nabla_\eta^{\Cal
    V}\ph(\eta_1,\dots,\eta_k))
$$
holds for arbitrary vector fields $\eta,\eta_1,\dots,\eta_k$. Using
the standard formula for $d^\nabla$, this readily implies that
$d^\nabla\o\Xi_*=\Xi_*\o d^\nabla$. Since we evidently get
$\Xi_*(\alpha\wedge i_\xi\ph)=\alpha\wedge i_\xi(\Xi_*\ph)$, this also
implies $\hat d\o\Xi_*=\Xi_*\o \hat d$.

(2) Theorem \ref{thm4.1} and the last part of Lemma \ref{lem4.3} show that
for $\ph\in K^k$, we get $d^\nabla\ph\in A^{k+1}_\hor$. By part (1),
we also get $\Xi_*(d^\nabla\ph)=d^\nabla(\Xi_*\ph)=0$, so indeed
$d^\nabla(K^k)\subset K^{k+1}$, and the last claim is obvious.

(3) For $\ps\in A^k_\hor$ we get $\hat d\Xi_*\ps=\Xi_*(\hat d\ps)$ by
part (1). But as observed above, $\hat d\ps\in A^{k+1}_\hor$, so we
conclude that $\hat d$ induces a well defined operator $C^k\to
C^{k+1}$. Next, for $\ph\in A^k_\hor$, the definition of $\hat d$ and
the last part of Lemma \ref{lem4.3} show that $\hat
d\ph=d^\nabla\ph-\al\wedge\Xi_*\ph$. The standard formula for
$d^\nabla$ easily implies that we can compute
$d^{\nabla}(\al\wedge\Xi_*\ph)$ as $d\al\wedge \Xi_*\ph-\al\wedge
d^\nabla(\Xi_*\ph)$. Since $\xi$ is the Reeb field for $\al$, the
first summand is horizontal already. On the other hand, the second
summand lies in the kernel of the projection to horizontal
forms, so $\hat d(\al\wedge\Xi_*\ph)=\Xi_*(d\al\wedge\ph)$. But
assuming $d^\nabla\o d^\nabla=0$, the fact that $\hat
d\ph=d^\nabla\ph-\al\wedge\Xi_*\ph$ implies that 
$$
\hat d\hat d\ph=-\hat d(\al\wedge\Xi_*\ph)=-\Xi_*(d\al\wedge\ph),
$$
so the last claim follows. 
\end{proof}

\subsection{A long exact sequence}\label{4.4} 
We are now ready to construct a long exact sequence of cohomology
groups, which will be the fundamental tool to compute the cohomology
of descended BGG sequences. To define the necessary maps, let us first
make the definition of the cohomology of $(C^*,\hat d)$ more
explicit. We assume that $d^\nabla\o d^\nabla=0$ from now on. A
$k$--cocycle in the complex $(C^*,\hat d)$ by definition is
represented by a form $\ph\in A^k_\hor$ for which there is a form
$\ps\in A^{k+1}_\hor$ such that $\hat d\ph=\Xi_*\ps$. We then simply
write $[\ph]\in H^k(C^*,\hat d)$ for the cohomology class represented
by $\ph$, and $[\ph]=[\tilde\ph]$ if and only if there are forms
$\ps_1\in A^{k-1}_\hor$ and $\ps_2\in A^k_\hor$ such that
$\tilde\ph=\ph+\hat d\ps_1+\Xi_*\ps_2$.

Now let us assume that $\tau\in A^k$ such that $d^\nabla\tau=0$. Then
by Lemma \ref{lem4.2} we get $i_\xi\tau\in A^{k-1}_\hor$ and Lemma
\ref{lem4.3} shows that $d^\nabla i_\xi\tau=\Xi_*\tau$. By definition,
this implies that $\hat di_\xi\tau=\Xi_*(\tau-\al\wedge i_\xi\tau)$,
so we can form the class $[i_\xi\tau]\in H^{k-1}(C^*,\hat d)$. We
obtain a map $\pi$ from $\ker(d^\nabla)\subset A^k$ to
$H^{k-1}(C^*,\hat d)$.

On the other hand, suppose that we have given $\ph\in A^{k-1}_\hor$
and $\ps\in A^k_\hor$ such that $\hat d\ph=\Xi_*\ps$. Then we can form
$d\al\wedge\ph+\hat d\ps$ and using that $\xi$ is the Reeb field for
$\al$, we see that this lies in $A^{k+1}_\hor$. Moreover by part (1)
of Proposition \ref{prop4.3} we get $\Xi_*(\hat d\ps)=\hat
d(\Xi_*\ps)=\hat d\hat d\ph$. In the proof of part (3) of that
Proposition, we have seen that $\hat d\hat d\ph=-\Xi_*(d\al\wedge\ph)$,
which shows that actually $d\al\wedge\ph+\hat d\ps\in K^{k+1}$.

By the last part of Lemma \ref{lem4.3}, $i_\xi d^\nabla\ps=\Xi_*\ps=\hat
d\ph$, so $\hat d\ps=d^\nabla\ps-\al\wedge\hat d\ph$, and in the last
term we can replace $\hat d\ph$ by $d^\nabla\ph$. Using this and
$(d^\nabla)^2=0$, we get $d^\nabla\hat d\ps=-d\al\wedge
d^\nabla\ph$. But this clearly cancels with $d^\nabla(d\al\wedge\ph)$,
so $d\al\wedge\ph+\hat d\ps$ is a cocycle in $K^{k+1}$ and we can form
the cohomology class $[d\al\wedge\ph+\hat d\ps]\in
H^{k+1}(K^*,d^{\nabla})$.

In the beginning we had fixed a form $\ps$ such that $\hat
d\ph=\Xi_*\ps$. Of course this pins down $\ps$ up to adding an element
of $K^k$. This shows that the cohomology class $[d\al\wedge\ph+\hat
d\ps]$ depends only on $\ph$, so we get a well defined map 
$$
\delta:\{\ph\in A^{k-1}_\hor: \hat d\ph\in \Xi_*(A^k_\hor)\}\to
H^{k+1}(K^*,d^{\nabla}).
$$ 

\begin{thm}\label{thm4.4}
  The maps $\pi$ and $\delta$ induce well defined maps in cohomology,
  which we denote by the same symbols, i.e.~$\pi:H^k(A^*,d^\nabla)\to
  H^{k-1}(C^*,\hat d)$ and $\delta:H^{k-1}(C^*,\hat d)\to
  H^{k+1}(K^*,d^\nabla)$. Together with the map $j$ induced by the
  inclusion $K^*\hookrightarrow A^*$, these fit into a long exact
  sequence of the form
$$
\dots\overset{\delta}{\to} H^k(K^*,d^\nabla)\overset{j}{\to}
H^k(A^*,d^\nabla) \overset{\pi}{\to} H^{k-1}(C^*,\hat d)
\overset{\delta}{\to} H^{k+1}(K^*,d^{\nabla})\overset{j}{\to}\dots
$$
\end{thm}
\begin{proof}
  Since both $\pi$ and $\delta$ are evidently linear, we have to show
  that they vanish on elements representing trivial cohomology classes
  to obtain well defined maps in cohomology. If $\tau\in A^{k-1}$,
  then Lemma \ref{lem4.3} shows that $i_\xi d^\nabla\tau=-d^\nabla
  i_\xi\tau+\Xi_*\tau$. Writing $-d^\nabla i_\xi\tau$ as $-\hat
  di_\xi\tau-\al\wedge i_\xi d^\nabla i_\xi\tau$, the second summand
  can be rewritten as $-\al\wedge \Xi_*(i_\xi\tau)$ by Lemma
  \ref{lem4.3}. Hence we see that $\pi(d^\nabla\tau)=\Xi_*(\tau-\al\wedge
  i_\xi\tau)$, and since $\tau-\al\wedge i_\xi\tau\in A^{k-1}_\hor$,
  this has trivial class in $H^k(C^*,\hat d)$.

  On the other hand, take $\ps_1,\ps_2\in A^*_\hor$ of degrees $k-2$
  and $k-1$ respectively. To determine $\de(\hat d\ps_1+\Xi_*\ps_2)$,
  we first have to compute the image of this element under $\hat
  d$. This gives $-d\al\wedge \Xi_*\ps_1+\Xi_*\hat d\ps_2$, so
$$
\de(d^\nabla\ps_1+\Xi_*\ps_2)=d\al\wedge
\hat d\ps_1+d\al\wedge\Xi_*\ps_2+\hat d(-d\al\wedge\ps_1+\hat d\ps_2).  
$$ 
Now the second and last term in the right hand side clearly cancel,
and a short computation shows that $ \hat
d(d\al\wedge\ps_1)=d\al\wedge\hat d\ps_1$, so the other two terms
cancel, too. This shows that $\delta$ induces a well defined map in
cohomology.

To prove exactness of the sequence, we first observe that for $\ph\in
K^k$, we have $i_\xi\ph=0$ by definition, so $\pi\o j=0$. On the other
hand, suppose that $\tau\in A^k$ satisfies $d^\nabla\tau=0$ and
$\pi([\tau])=0$. Then $i_\xi\tau=\hat d\ps_1+\Xi_*\ps_2$ for elements
$\ps_1,\ps_2\in A^*_\hor$ of degree $k-2$ and $k-1$,
respectively. Then consider the form
$$
\tilde\tau:=\tau+d^\nabla(\al\wedge\ps_1-\ps_2)=\tau+d\al\wedge\ps_1-\al\wedge d^\nabla\ps_1-d^\nabla\ps_2,
$$ 
which represents the same cohomology class as $\tau$. Under insertion
of $\xi$, the second summand in the right hand side vanishes, while
the third summand produces $-d^\nabla\ps_1+\al\wedge i_\xi
d^\nabla\ps_1=-\hat d\ps_1$ and the last one gives
$-\Xi_*\ps_2$. Hence $i_\xi\tilde\tau=0$ and since also
$d^\nabla\tilde\tau=0$, Lemma \ref{lem4.3} shows that
$\Xi_*\tilde\tau=0$. Hence $\tilde\tau$ is a cocycle in $K^k$ and
$\ker(\pi)=\im(j)$.

Next, we claim that $\delta\o\pi=0$. Taking $\tau\in A^k$ with
$d^\nabla\tau=0$, we have observed above that $\hat d
i_\xi\tau=\Xi_*(\tau-\al\wedge i_\xi\tau)$. So by definition,
$\delta([i_\xi\tau])$ is the cohomology class of $d\al\wedge
i_\xi\tau+\hat d(\tau-\al\wedge i_\xi\tau)$. Computing
$d^\nabla(\tau-\al\wedge i_\xi\tau)$ using that $\tau$ is closed, we
get $-d\al\wedge i_\xi\tau+\al\wedge d^\nabla i_\xi\tau$. Projecting to
the horizontal part leaves the first term untouched and kills the
second term, so $\delta([i_\xi\tau])=0$.

Conversely, let us assume that $\ph\in A^{k-1}_\hor$ has the property
that $\hat d\ph=\Xi_*\ps$ for some $\ps\in A^k_\hor$ and that
$\delta([\ph])=0$. This means that there is $\tilde\ps\in K^k$ such
that $d\al\wedge \ph+\hat d\ps=d^\nabla\tilde\ps$. Taking into account
that $i_\xi d^\nabla\tilde\ps=0$, we may simply replace $\ps$ by
$\ps-\tilde\ps$, and assume that $\hat d\ph=\Xi_*\ps$ and $d\al\wedge
\ph+\hat d\ps=0$. Now consider $\tau:=\al\wedge\ph+\ps\in A^k$, which
evidently satisfies $i_\xi\tau=\ph$. Now 
$$
d^\nabla\tau=d\al\wedge\ph-\al\wedge d^\nabla\ph+d^\nabla\ps
$$
Now in the second summand, we can clearly replace $d^\nabla$ by $\hat
d$. In the third summand, we rewrite $d^\nabla\ps=\hat d\ps+\al\wedge
i_\xi d^\nabla\ps$. Rewriting the last term as $\al\wedge\Xi_*\ps$ we
conclude that $d^\nabla\tau=0$, so $[\ph]=\pi([\tau])$ and
$\ker(\delta)=\im(\pi)$.

Finally, if $\ph\in A^{k-1}_\hor$ has the property that $\hat
d\ph=\Xi_*\ps$ for some $\ps\in A^k_\hor$, then
$\delta([\ph])=d\al\wedge\ph+\hat d\ps$. But then
$\tau:=\al\wedge\ph+\ps\in A^k$ and we get 
$$
d^\nabla\tau=d\al\wedge\ph-\al\wedge d^\nabla\ph+\hat d\ps+\al\wedge
\Xi_*\ps.
$$
This vanishes since in the second term we may replace $d^\nabla$ by
$\hat d$, and we see that $j\o\delta=0$. 

Conversely, assume that $\ph\in K^{k+1}$ has the property that
$\ph=d^\nabla\tau$ for some $\tau\in A^k$. Then by assumption, we have
$0=i_\xi d^\nabla\tau$ and Lemma \ref{lem4.3} shows that $d^\nabla
i_\xi\tau=\Xi_*\tau$. Thus we get $\hat d
i_\xi\tau=\Xi_*(\tau-\al\wedge i_\xi\tau)$ and we may form
$[i_\xi\tau]\in H^{k-1}(C^*,\hat d)$. But then $\delta([i_\xi\tau])$
is the class of $d\al\wedge i_\xi\tau-\hat d(\tau-\al\wedge
i_\xi\tau)$. Now $\hat d\tau=d^\nabla\tau=\ph$, while
$d^\nabla(\al\wedge i_\xi\tau)=d\al\wedge i_\xi\tau-\al\wedge d^\nabla
i_\xi\tau$. As before, projecting the right hand side to the
horizontal part leaves the first term unchanged and kills the second
term, so $\delta([i_\xi\tau])=[\ph]$, which completes the proof.
\end{proof}

\subsection{The case of the homogeneous model}\label{4.5} 
As a last step, we specialize further to the case of PCS--quotients
of connected open subsets of the homogeneous model $G/P$ of a
parabolic contact structure. In particular, this includes the global
contactification $S^{2n+1}\to\Bbb CP^n$ in the two geometric
interpretations discussed in Proposition 1 and Section 3.4 of
\cite{PCS2}. In the second case, we obtain generalizations of all
results on cohomology needed for the applications in \cite{E-G}.

Observe that all restrictions we have imposed so far are satisfied for
a connected open subset in $G/P$, since any infinitesimal automorphism
of a locally flat geometry is normal. The crucial additional
ingredient we get for the homogeneous model is that any tractor bundle
admits a global parallel frame.

\begin{lemma}\label{lem4.5}
  Let $M^\#=G/P$ be the homogeneous model of a parabolic contact
  structure, let $\Bbb V$ be a representation of $G$ and $\Cal
  VM^\#=G\x_P\Bbb V$ the corresponding tractor bundle. For $v\in\Bbb
  V$ consider the section $\si_v\in\Ga(\Cal VM^\#)$ corresponding to
  the $P$--equivariant function $f_v:G\to\Bbb V$ defined by
  $f_v(g):=g^{-1}\cdot v$. Then $\nabla^{\Cal V}s_v=0$, so starting
  from a basis of $\Bbb V$, we obtain a global parallel frame for
  $\Cal VM^\#$.
\end{lemma}
\begin{proof}
  It is clear that each $f_v$ is equivariant and that the values of
  the $\si_v$ in each point fill the whole fiber, so we only have to
  show that each $s_v$ is parallel. Applying the construction to the
  adjoint tractor bundle $\Cal AM^\#$, we associate to $X\in\frak g$
  the global section $s_X$ corresponding to the function $g\mapsto
  \Ad(g^{-1})\cdot X$. By equivariancy of the Maurer--Cartan form,
  this corresponds to the right--invariant vector field $R_X$
  generated by $X$. 

  Computing $(R_X\cdot f_v)(g)$ as the derivative at $t=0$ of
  $f_v(\exp(tX)g)$ immediately shows that $R_X\cdot f_v=f_{-X\cdot
    v}$, where in the right hand side we use the infinitesimal action
  of $\frak g$ on $\Bbb V$. The description of the tractor connection
  in terms of the fundamental derivative used in the proof of Lemma
  \ref{lem4.3} then readily implies that $\si_v$ is parallel along the
  projection of $R_X$. Since any tangent vector on $M^\#$ can be
  realized as such a projection, we get $\nabla^{\Cal V}s_v=0$ and the
  result follows.
\end{proof}

Now of course we also get a global parallel trivialization of $\Cal
VM^\#$ in the case that $M^\#$ is a connected open subset in $G/P$. In
this case $\Cal G^\#\subset G$ is the (open) pre--image of
$M^\#\subset G/P$ in $G$. Now assume that $q:M^\#\to M$ is a
PCS--quotient, and let $\tilde\xi\in\frak X(M^\#)$ be the
corresponding infinitesimal automorphism. Since any infinitesimal
automorphism corresponds to a parallel section of $\Cal AM$, we see
that $\tilde\xi$ must be the restriction to $\Cal G^\#$ of a right
invariant vector field $R_X$ on $G$. From the proof of Lemma \ref{lem4.5}
above, we see that the corresponding bundle map $\Xi$ on $\Cal VM^\#$
satisfies $\Xi\o \si_v=\si_{-X\cdot v}$. Now let $\Bbb W_1\subset\Bbb
V$ be the kernel and $\Bbb W_2$ the cokernel of the map $\Bbb V\to\Bbb
V$ defined by $v\mapsto X\cdot v$. Then of course mapping $(x,w)$ to
$s_w(x)$ defines a trivialization $M^\#\x\Bbb W_1\cong\ker(\Xi)$ and
similarly we get a trivialization of $\Coker(\Xi)$.

\begin{thm}\label{thm4.5}
  Suppose that $M^\#$ is a connected open subset of the homogeneous
  model $G/P$ of some parabolic contact structure and that $q:M^\#\to
  M$ is a PCS--quotient for which the infinitesimal automorphism
  defining the quotient corresponds to $X\in\frak g$. Let $\Bbb V$ be
  a representation of $\frak g$, let $\rho_X:\Bbb V\to\Bbb V$ be the
  action of $X$ and put $\Bbb W_1:=\ker(\rho_X)$ and $\Bbb W_2:=\Bbb
  V/\im(\rho_X)$.

  (1) For the complex $(K^*,d^\nabla)$ from Proposition \ref{prop4.3}, the
  cohomology is given by $H^k(K^*,d^\nabla)\cong H^k(M)\otimes\Bbb
  W_1$, where $H^k(M)$ is the $k$--th de--Rham cohomology of $M$.

  (2) Suppose further that $\ker(\rho_X)\cap\im(\rho_X)=\{0\}$. Then
  there is a natural isomorphism $\Bbb W_1\cong\Bbb W_2$ and also for
  the complex $(C^*,\hat d)$ from Proposition \ref{prop4.3}, the
  cohomology is given by $H^k(C^*,\hat d)\cong H^k(M)\otimes\Bbb
  W_1$. Moreover, under this identification and the one from part (1),
  the homomorphism $\delta$ in the long exact sequence from Theorem
  \ref{thm4.4} corresponds to map $H^{i-1}(M)\otimes\Bbb W_1\to
  H^{i+1}(M)\otimes\Bbb W_1$ is given by taking the wedge product with
  the cohomology class $[\om]\in H^2(M)$, where $\om\in\Om^2(M)$ is
  characterized by $q^*\om=d\al$.
\end{thm}
\begin{proof}
  The global trivialization $\Cal VM^\#\cong M^\#\x\Bbb V$ constructed
  above of course defines an isomorphism 
\begin{equation}\label{iso}
\Om^k(M^\#,\Cal VM^\#)\cong\Om^k(M^\#)\otimes\Bbb V.
\end{equation} 
By definition, the map $\Xi_*$ corresponds to $\id\otimes\rho_X$ under
this isomorphism while $i_\xi$ corresponds to $i_\xi\otimes\id_{\Bbb
  V}$. Moreover, the fact that the trivializing frame consists of
parallel sections implies that $d^\nabla$ corresponds to
$d\otimes\id_{\Bbb V}$ under the isomorphism \eqref{iso}. Finally, the
considerations about naturality of $\Cal L_{\tilde\xi}$ from Section
\ref{4.2} together with the observations on the trivializing sections
above show that $\Cal L_{\tilde\xi}$ corresponds to $\Cal
L_\xi\otimes\id-\id\otimes\rho_X$ under the isomorphism \eqref{iso}.

  Now by definition $K^k\subset \Om^k(M^\#,\Cal VM^\#)$ consists of
  those forms $\ph$ such that $\Cal L_\xi\ph=0$, $i_\xi\ph=0$ and
  $\Xi_*(\ph)=0$. Hence we see restricting the above map, we obtain an
  isomorphism between $K^k$ and the joint kernel of
  $\id\otimes\rho_X$, $i_\xi\otimes\id$ and  $\Cal
  L_\xi\otimes\id$. Of course, this joint kernel is exactly
  $\Om^*(M)\otimes\Bbb W_1$, and $d^\nabla$ corresponds to
  $d\otimes\id$, so (1) follows. 

  In the setting of (2), we first observe that restricting the
  projection $\Bbb V\to\Bbb W_2$ to $\Bbb W_1$, we obtain an injection
  by assumption, so this must be a linear isomorphism for dimensional
  reasons. Now we can compose the isomorphism \eqref{iso} with the
  projection onto $\Om^k(M^\#)\otimes\Bbb W_2$ and restrict the
  resulting map to $A^k_{hor}\subset \Om^k(M^\#,\Cal VM^\#)$. By the
  above observations on compatibility, the values of this map lie in
  the kernels of $\Cal L_\xi\otimes\id$ and $i_\xi\otimes\id$, so we
  actually land in $\Om^k(M)\otimes\Bbb W_2$. 

  Moreover, by assumption, any form in $\Om^k(M^\#,\Cal VM^\#)$ can be
  written as $\ph=\ph_1+\ph_2$, where $\ph_1$ has values in
  $\ker(\Xi)$ while $\ph_2$ has values in $\im(\Xi)$. From above, we
  see that $\Cal L_{\tilde\xi}$ preserves these two subspaces, so we
  see that $\Cal L_{\tilde\xi}\ph=0$ if and only if $\Cal
  L_{\tilde\xi}\ph_i=0$ for $i=1,2$. The same result trivially holds
  for $i_\xi$ so we see that $\ph\in A^k_{hor}$ implies $\ph_i\in
  A^k_{hor}$ for $i=1,2$, so in particular $\ph_1\in K^k$. Again by
  assumption $\Xi$ restricts to an isomorphism on $\Im(\Xi)$, which
  shows that $\ph_2\in\Xi_*(A^k_{hor})$, so the class of $\ph$ in
  $C^k$ coincides with the class of $\ph_1$.

  On the other hand, given $\tau\in\Om^k(M)$ and $w\in\Bbb W_2$, we
  can find an element $\tilde w\in\Bbb W_1$ projecting onto $w$ and
  the consider $q^*\tau\otimes\si_{\tilde w}\in\Om^k(M^\#,\Cal
  VM^\#)$. Since $\Xi_*(\si_{\tilde w})=0$, we see that this lies in
  $A^k_\hor$, so we can look at its class in $C^k$. Together with the
  above, this shows that we get an inverse, so
  $C^k\cong\Om^k(M)\otimes\Bbb W_2$. Of course,
  $d^\nabla(q^*\tau\otimes\si_{\tilde w})=(q^*d\tau)\otimes\si_{\tilde
    w}$, and since the pullback is horizontal, this coincides with
  $\hat d(q^*\tau\otimes\si_{\tilde w})$. Hence under our isomorphism
  $\hat d$ on $C^*$ again corresponds to $d\otimes\id$. Finally, if
  $d\tau=0$, then $\hat d(q^*\tau\otimes\si_{\tilde w})=0$, which
  readily implies the claim about $\delta$.
\end{proof}

\subsection{Examples}\label{4.6}
Let us first observe that the conditions of part (2) of Theorem
\ref{thm4.5} are often satisfied.

\begin{prop}\label{prop4.6}
  Let $\frak g$ be a simple Lie algebra with complexification $\frak
  g_{\Bbb C}$ and suppose that $X\in\frak g\subset\frak g_{\Bbb C}$ is
  semisimple, i.e.~such that $\ad_X$ is diagonalizable on $\frak
  g_{\Bbb C}$. Then the assumption of part (2) of Theorem \ref{thm4.5} is
  satisfied for any finite dimensional representation of $\frak g$.
\end{prop}
\begin{proof}
  It is a classical result of Lie theory that $X$ acts diagonalizably
  on any complex representation of $\frak g_{\Bbb C}$. But for a
  diagonalizable map, the kernel is the eigenspace for the eigenvalue
  $0$, while the image coincides with the sum of all other
  eigenspaces. Hence $\ker(\rho_X)\cap\im(\rho_X)=\{0\}$ on such
  representations, and via complexifications, this easily extends to
  all real representations of $\frak g$.
\end{proof}

Next, we can sort out the local case. 

\begin{cor}\label{cor4.6}
  Suppose that the assumptions of part (2) of Theorem \ref{thm4.5} are
  satisfied and that $q:M^\#\to M$ has the property that the form
  $\om\in\Om^2(M)$ such that $q^*\om=d\al$ is exact. Then for each
  $k$, the cohomology $\Cal H_k$ in degree $k$ of the descended BGG
  sequence fits into an exact sequence
$$
0\to \Om^k(M)\otimes\Bbb W_1\to \Cal H_k\to
\Om^{k-1}(M)\otimes\Bbb W_1\to 0. 
$$
In particular, the local cohomology of the complex vanishes except in
degrees $0$ and $1$, where it is isomorphic to $\Bbb W_1$.
\end{cor}
\begin{proof}
  Part (2) of Theorem \ref{thm4.5} gives an interpretation of the
  cohomology groups of $K^*$ and $C^*$ showing up in the long exact
  sequence from Theorem \ref{thm4.4} and shows that the connecting
  homomorphisms $\delta$ in that sequence are all $0$. Hence the
  sequence decomposes into short exact sequences as claimed. The
  result on local cohomology follows readily.
\end{proof}

Finally, we can sort of the case of complex projective space in either
of the two interpretations from \cite{PCS2}. Note that the only
information needed for the applications in \cite{E-G} is vanishing of
the first cohomology for a class of descended BGG sequences. 

\begin{thm}\label{thm4.6}
  For $n\geq 2$ consider the global PCS--quotient
  $q:M^\#:=S^{2n+1}\to\mathbb CP^n=:M$, either for the PCS--structure
  of K\"ahler type on $\mathbb CP^n$ as discussed in Proposition 1
  of \cite{PCS2} or the induced conformal Fedosov structure as in
  Section 3.4 of that reference. Let $\mathbb V$ be a representation
  of the corresponding group $G$, let $X\in\frak g$ be the element
  generating the parallel section of $\Cal AM^\#$ giving rise to the
  PCS--quotient and put $\Bbb W:=\{v\in\Bbb V:X\cdot v=0\}$. Let $\Cal
  VM^\#\to M^\#$ be the tractor bundle induced by $\Bbb V$.

  Then the cohomology of the sequence of differential operators on $M$
  obtained by descending the BGG sequence induced by $\Bbb V$ vanishes
  in degrees different from $0$ and $2n+1$, while in degrees $0$ and
  $2n+1$ it is isomorphic to $\Bbb W$. 
\end{thm}
\begin{proof}
  The Lie algebra $\frak g$ of $G$ either equals
  $\mathfrak{su}(n+1,1)$ or $\mathfrak{sp}(2n+2,\mathbb R)$. In the
  first case, $\frak g$ naturally acts on $\Bbb C^{n+2}$ and in the
  second case we consider it as acting on $\Bbb C^{n+1}\cong\Bbb
  R^{2n+2}$. In both cases, the discussion in \cite{PCS2} shows that
  $X$ acts diagonalizably (over $\mathbb C$) in this
  representation. Thus we can apply Proposition \ref{prop4.6} to see that
  the assumptions of part (2) of Theorem \ref{thm4.5} are satisfied for
  $\Bbb V$. Together with the well known description of $H^*(\Bbb
  CP^n)$, Theorem \ref{thm4.5} shows that both $(K^*,d^\nabla)$ and
  $(C^*,\hat d)$ have vanishing cohomology in odd degrees. Moreover,
  all the connecting homomorphisms $\delta:H^{k-1}(K^*,d^\nabla)\to
  H^{k+1}(C^*,\hat d)$ in the long exact sequence from Theorem
  \ref{thm4.4} are isomorphisms whenever $1\leq k\leq 2n-1$. Using this,
  the long exact sequence readily implies vanishing of the cohomology of
  $(A^*,d^\nabla)$ in degrees different from $0$ and $2n+1$. For these
  two degrees the long exact sequence contains the parts $0\to
  H^0(K^*,d^\nabla)\to H^0(A^*,d^\nabla)\to 0$ and $0\to
  H^{2n+1}(A^*,d^\nabla)\to H^{2n}(C^*,\hat d)\to 0$ which together
  with Theorem \ref{thm4.5} completes the proof.
\end{proof}

\end{document}